\documentclass[11pt,reqno]{amsart}

\usepackage[active]{srcltx}
\usepackage{graphicx}
\usepackage{color}
\usepackage{enumerate,cleveref}
\usepackage{amssymb}
\usepackage{overpic}
\usepackage{amsfonts,amsmath,amssymb}
\usepackage{epstopdf}
\usepackage{algorithmic}
\usepackage{mathrsfs}
\usepackage{amsopn}
\usepackage{bm}
\usepackage[ruled]{algorithm2e}
\usepackage{array,booktabs}
\usepackage{subcaption}
\usepackage{soul}

\newtheorem{lemma}{Lemma}[section]
\newtheorem{corollary}[lemma]{Corollary}
\newtheorem{theorem}[lemma]{Theorem}
\newtheorem{proposition}[lemma]{Proposition}

\newtheorem{definition}[lemma]{Definition}

\DeclareMathOperator*{\argmin}{arg\,min}
\DeclareMathOperator\supp{supp}

\makeatletter
\@namedef{subjclassname@1991}{Subject Classification}
\makeatother

\title[Structured IHT with On- and Off-Grid Applications]{Structured Iterative Hard Thresholding with On- and Off-Grid Applications}

\author{Joseph S.~Donato}
\address{ University of Michigan \\ Ann Arbor, MI 48109}
\email{jsdonato@umich.edu}
\author{Howard W.~Levinson}
\address{Department of Mathematics and Computer Science\\ Santa Clara University \\ Santa Clara, CA 95050}
\email{hlevinson@scu.edu}

\begin{document}

\begin{abstract}
We consider linear sparse recovery problems where additional structure regarding the support of the solution is known. The form of the structure considered is non-overlapping sets of indices that each contain part of the support.  An algorithm based on iterative hard thresholding is proposed to solve this problem.   The convergence and error of the method are analyzed with respect to mutual coherence.  Numerical simulations are examined in the context of an inverse source problem, including modifications for off-grid recovery.  
\end{abstract}



\maketitle
\smallskip

\section{Introduction}
Sparse regularization is a powerful tool for recovering  solutions to linear problems where the true solution is known {\it a priori} to be sparse.  This has been a topic of high interest, especially in the case of underdetermined systems.  This has produced extensive literature in the field of compressive sensing \cite{candes2006stable,donoho2006compressed}.  Many effective algorithms have been developed  to solve linear sparse recovery problems, with accompanying theoretical analyses \cite{chen2001atomic,tibshirani1996regression,tropp2007signal}.
This  includes the iterative hard thresholding algorithm \cite{blumensath2009iterative}.

For some linear sparse recovery problems, additional information is known about the structure of the sparsity.  This known information regarding the support of the solution is referred to as structured sparsity, and can be quite general in form \cite{bach2012structured,micchelli2013regularizers,yu2011solving,yuan2006model}.  Algorithms for structured sparsity take advantage of this additional information to improve reliability or convergence speed  \cite{boyer2019compressed,li2019compressed}.

Many linear inverse problems with applications in the applied sciences benefit from sparse regularization \cite{candes2007sparsity,daubechies2004iterative,jin2017sparsity}.   This is true for direction-of-arrival (DOA) and related inverse source problems \cite{kim2012compressive,zhang2013sparsity}.  As an example application, we consider the related problem of determining the incident angles and real amplitudes of a small number of simultaneously incoming plane waves using measurements at known detector locations.  By applying structured sparsity ideas to these linear inverse problems, recovery is much more likely.

The mutual coherence, or coherence, of a matrix can be a useful tool for studying the convergence of sparse recovery algorithms \cite{donoho2003optimally,donoho2005stable,tropp2006just}.  Defined as the maximum absolute value of the cross-correlations of the columns of $A$, the coherence yields worst case bounds for many algorithms.  
However, despite its ease of calculation, coherence is often an unhelpful measure for analyzing many linear inverse problems that come from physical processes.  This is often due to the fact that a fine grid may be required to accurately discretize the physical process, even if the true support of the sparse solution is well separated.  This is indeed the case for the considered inverse source problem.  Off-grid methods are often desired in order to find the true incident angles regardless of an initial grid.  Off-grid sparsity algorithms are very useful for these cases, with off-grid algorithms studied for DOA problems in \cite{tan2013sparse,wu2018two,yang2018sparse, yang2012off}.  One hope with these techniques is that they yield reduced coherence values which align better with theory.   

In this paper, we introduce and develop a generalized form of the IHT algorithm for cases where specific information on the structure of the sparsity is known. This information will be in the form of mutually disjoint index sets which each reference a certain number of nonzero elements.  We refer to this algorithm as the structured IHT algorithm. The main contributions of the paper are to provide analysis and numerical simulations that demonstrate the benefit of the structured IHT algorithm over IHT when this additional structure is known.  We also develop a related algorithm for solving discretized linear sparse recovery problems where the true solution is off-grid. 

We mention the following additional related works.  Other examples of investigating the limitations of coherence can be found in \cite{adcock2017breaking}, and, especially in regards to linear inverse problems, \cite{jones2015asymptotic}.   While structured sparsity can be quite general, the specific form of structured sparsity considered in this paper is equivalent (up to permutation of entries) to the idea of sparsity in levels found in \cite{adcock2017breaking,doi:10.1137/15M1043972}. Recently, the works of \cite{IHTL,Adcock_2021} independently introduced an analogous version to the structured IHT algorithm proposed in our work, called IHT in Levels (IHTL).  In the cited work, the algorithm is analyzed via the restricted isometry property (RIP), as opposed to coherence level analysis.  A stochastic version of a structured IHT type algorithm was introduced in  \cite{zhou2019stochastic}. The work of \cite{NIPS2016_8e82ab72} considers IHT-type algorithms for recovering vectors that have overlapping structured sparsity.  This overlapping structure increases the complexity of the thresholding operator, which is a main focus of the cited work.  

The paper is organized as follows. In section 2, we review the  IHT algorithm and introduce the structured IHT algorithm for solving structured sparse linear recovery problems.  In section 3, we analyze the proposed algorithm by means of mutual coherence and provide theoretical guarantees of convergence.  We apply the structured IHT algorithm to a specific inverse source problem in section 4.  These results are compared with the preceding theory.   In section 5, we introduce a related algorithm for solving the same problem but with off-grid targets.  Numerical simulations and theoretical justifications are provided for this algorithm within the context of off-grid recovery.  We conclude with a discussion in section 6.

\section{Background and the Structured IHT Algorithm}
Throughout the paper, matrices will be denoted by uppercase letters and vectors by lowercase letter.  For a matrix $A$ with entries $A_{ij}$, we denote its $j$th column as $A_j$.  For a vector $x$, $\|x\|_p$ will denote its standard $\ell_p$ norm.   If $x$ is a vector of length $N$, and $S\subseteq\{1,...,N\}$, then $x_S$ is the vector of length $|S|$ containing the entries of $x$ restricted to $S$.

\subsection{Iterative Hard Thresholding Algorithm}

The original IHT algorithm was introduced in \cite{blumensath2009iterative} as a sparsity promoting recovery algorithm that was easy to implement and had   theoretical guarantees.  This algorithm has been extensively developed \cite{blanchard2015cgiht,blumensath2012accelerated,
blumensath2008iterative,blumensath2010normalized,jain2014iterative}.  The algorithm addresses the linear sparse recovery problem
\begin{equation}
\label{eq:basic_lineq}
Ax=b +\epsilon\ ,
\end{equation}
where $A$ is an $M\times N$ matrix, $\epsilon$ is a vector containing any noise, and $x$ is known to have only $k$ nonzero elements, where typically $k\ll N$.  This problem can be overdetermined $(M>N)$, but is typically assumed to be underdetermined in compressive sensing literature.  The algorithm solves for the $k$-sparse vector $x$ using the iterative update
\begin{equation}
\label{eq:IHT}
x^{(t+1)}=H_k \left( x^{(t)} + A^*(b - A x^{(t)} ) \right),
\end{equation}
where $H_k$ is the hard thresholding operator which sets all but the $k$ largest elements of $x$ (in absolute value) to 0.  This algorithm is the classic Richardson first-order iteration with the additional application of a thresholding operator after each iteration to promote sparsity.

\subsection{Structured IHT Algorithm}
The IHT algorithm updates the vector using a {\it global} thresholding operator $H_k$, but what if further information regarding the location of the nonzero entries was known? Structured sparsity is a well studied concept, especially in statistical learning theory.  The main concept is to use additional information about the structure of the sparsity, beyond the fact that $k$ entries are nonzero.  This structure is often in the form of group sparsity, an underlying graph, or a hierarchical method \cite{kim2010tree,maurer2012structured}.  This information could be additional {\it a priori} information, or it could be obtained by a ``preprocessing" step.  This latter case will be explored in Section \ref{sec:ongrid}.    

 For this paper, we will assume this known structure takes a specific form which groups the indices of $x$ into disjoint sets.  While the IHT algorithm assumes there are $k$ nonzero elements anywhere in $x$, we will assume that there are $L$ known disjoint index sets $S_1,\dots S_L$ such that $x_{S_j}$ (the vector $x$ restricted to indices in $S_j$), is $k_j$ sparse.  That is, we have a collection of index sets such that
\begin{align}
\label{eq:ss_def}
S_i\cap S_j =\emptyset \quad \forall i,j \ ; \
\bigcup_{j=1}^L S_j \subseteq \mathcal{I}_N= \{1,\dots,N\}\ ; \
\|x_{S_j}\|_0=k_j \text{ with } \sum_{j=1}^L k_j  = k \ ,
\end{align}
where the $\ell_0$ ``norm" counts the number of nonzero elements in the vector.
We remark that, up to a permutation of indices, this definition of structured sparsity is equivalent to the definition of sparsity in levels in \cite{Adcock_2021}.  Sparsity in levels assumes that the sets $S_j$ contain a connected interval of indices, which is highly applicable to vectors in specific bases (such as wavelets, for example).  We will use the slightly more general form of \eqref{eq:ss_def} in this paper.   

With this additional knowledge of the structure of $x$, it is natural to introduce a set of {\it local} thresholding operators $H^{(j)}_{k_j}$, defined by
\begin{equation}
\left[H^{(j)}_{k_j}(x)\right]_{S_j}= 
H_{k_j}(x_{S_j})\ \quad \ , \ \quad \left[H^{(j)}_{k_j}(x)\right]_{\mathcal{I}_N\setminus S_j} = x_{\mathcal{I}_N\setminus S_j} \ .
\end{equation}
By this definition, $H^{(j)}_{k_j}$ leaves entries outside of $S_j$ unchanged, and thresholds the entries in $S_j$ to keep only the $k_j$ largest.  
To threshold the entire vector $x$ properly, one applies each local thresholding operator to $x$ 
\begin{equation}
\mathcal{H}^{(S)}_k(x):=\left(H_{k_1}^{(1)}\circ\cdots\circ H_{k_L}^{(L)}\right)(x)
\end{equation}
Note that as the index sets $S_1,\dots,S_L$ are disjoint, the operators $H^{(1)},\dots,H^{(L)}$ commute.
The structured IHT algorithm to solve \eqref{eq:basic_lineq} is now iteratively defined by
\begin{equation}
\label{eq:siht}
x^{(t+1)}=\mathcal{H}^{(S)}_k\left( x^{(t)} + A^*(b - A x^{(t)} ) \right) \ .
\end{equation}
While use of the structured IHT algorithm requires additional information (in the form of sparsity levels for subsets of $x$), we are able to use this information to obtain stronger theoretical guarantees.  Practically speaking, the structured IHT algorithm should converge faster than IHT.  It can be highly effective when the sets $S_j$ contain entries that are unrelated in regards to the sensing matrix $A$.  Pseudocode for the structured IHT algorithm is contained in Algorithm \ref{algo:siht}.

\begin{algorithm}[H]
\label{algo:siht}
\SetAlgoLined
 {\bf Input}: Sensing matrix $A$ and data $b$\\
 {\bf Input}: Index sets $S_1,\dots,S_L$\\
 {\bf Input}: Sparsity levels $k_1,\dots,k_L$\\
 {\bf Output}: $x$\\
 \vspace{3mm}
 Initialize $x^{(0)}, t=0$\\
 \While{stopping criteria not met}{
  $z^{(t)} = x^{(t)} + A^*(b - A x^{(t)} )$\\
 \For{j=1 : L}{
$x^{(t+1)}=H_{k_j}^{(j)}(z^{(t)})$ 
  }
  $t\gets t+1$
 }
 \caption{Structured Iterative Hard Thresholding}
\end{algorithm} 
\vspace{1pc}
The computational complexity of the algorithm is the same as the standard IHT algorithm, where each iteration requires $O(MN)$ operations, obtained by first multiplying $Ax^{(t)}$, and subsequently multiplying this result by $A^*$ (as opposed to precomputing the matrix product $A^*A$).  There is possibly additional storage required, $O(N)$, to store the index sets $S_j$.  As the local thresholding operators commute, the inner loop can be run in parallel, though for small $L$, it may be more efficient to run them sequentially (in any order).

\section{Analysis of the Algorithm}
\label{sec:analysis}
Analysis of sparse recovery algorithms usually rely on one of two quantities: the restricted isometry property (RIP) or coherence.  The RIP measures how ``close" the sensing matrix $A$ is to being an orthogonal matrix.  As defined and developed in \cite{CandesTao2005}, a matrix $A$ satisfies the RIP of order $k$ if there exists a constant $\delta_k$ with $0 < \delta_k < 1$
such that, for all $k$-sparse vectors $x$, 
\begin{equation}
	(1-\delta_k)\|x\|_2^2\le \|Ax\|_2^2 \le (1+\delta_k)\|x\|_2^2 \ .
\end{equation}
While matrices that satisfy RIP bounds lead to strong theoretical guarantees, verifying that the RIP bound holds is a nontrivial task \cite{calderbank2010construction,rudelson2008sparse}.  Furthermore, many deterministic sensing matrices of interest do not obey strong RIP constraints \cite{bandeira2013road}.  

In contrast, the coherence of a matrix is a quantity that is easily calculated, with the downside of weaker theoretical results \cite{tillmann2013computational}.  In some sense, coherence-based analysis considers situations which are worst case scenarios.  The coherence $\mu(A)$ of a matrix is defined as
\begin{equation}
\label{eq:mu_def}
\mu(A)=\max_{i\neq j} \frac{|\langle A_i, A_j\rangle|}{\|A_i\|_2\|A_j\|_2} \ ,
\end{equation}
where we recall $A_i$ is the $i$th column of $A$. 
As opposed to the RIP constant, the coherence is an inherent property of the sensing matrix $A$ and does not depend on the underlying sparsity of $x$.  We note that the coherence, which can be easily computed, yields upper and lower bounds on the RIP constant \cite{5075882}, which can be helpful in specific cases.  However, these bounds are often not tight, and RIP analysis is often performed separately from a coherence analysis \cite{5484999}.

A version of the structured IHT algorithm has recently been analyzed via RIP in \cite{IHTL,Adcock_2021}.  Our coherence analysis of the structured IHT algorithm builds upon existing coherence analysis for the IHT algorithm.  The following theorem gives theoretical guarantees for the original IHT algorithm, which depends on the coherence and the sparsity of $x$.
\begin{theorem}
\label{thm1}
(Adapted from Theorem 3 in \cite{wang2015linear}) \\
Let $\{x^{(t)}\}$ be the sequence generated from the IHT algorithm given by \eqref{eq:IHT} for the equation $Ax^*=b+\epsilon$.  If $x^*$ is $k$-sparse, then 
\begin{equation}
\label{eq:result_iht}
\|x^{(t)}-x^*\|_1\le (3\mu k)^t\|x^{(0)}-x^*\|_1 +\frac{3k}{1-3\mu k}\|A^*\epsilon\|_\infty \ .
\end{equation}
\end{theorem}
This theorem implies linear convergence if $3\mu k <1$.  Practically speaking, $\mu<1/(3k)$ often holds for only very small values of $k$.  This is the main motivation for the following analysis of the structured IHT algorithm.  The hope is that, roughly speaking, the coherence only has to be smaller than $1/(3k_{\max})$ where $k_{\max}=\max_j k_j\le k$.
To fully analyze the structured IHT algorithm, we introduce a variant of coherence that will be applicable to our restrictions to certain index sets.  
\begin{definition}
We define the restricted coherence $\mu_{S,S'}$ between two sets $S$ and $S'$ of a matrix $A$ by
\begin{equation}
\mu_{S,S'}(A)=\max_{\substack{i\in S , j\in S'\\i\neq j}}\frac{\langle A_i,A_j \rangle}{\|A_i\|_2\|A_j\|_2} \ .
\end{equation}
When $S=S'$, $\mu_{S,S}$ will be denoted simply as $\mu_S(A)$.  
\end{definition}
Note that by this definition, $\mu_{S,S'}(A)$ is less than or equal to  $\mu(A)$ for any sets $S$ and $S'$.  For many physical processes that can be represented by the matrix $A$, if $S$ and $S'$ are index sets representing sufficiently different entries (for example, in physical location), one can expect that $\mu_{S,S'}(A)\ll \mu(A)$.  This idea will be further explored in the application of structured IHT to an inverse source problem in Section \ref{sec:ongrid}.  We will write $\mu_{S,S'}(A)$ and $\mu_S(A)$ as simply $\mu_{S,S'}$ and $\mu_S$ respectively when the dependence on the matrix $A$ is clear.

We are now ready for the analysis of the structured IHT algorithm.  Suppose the true sparsity value $k$ is known, as well as $L$ index sets, $S_1,\dots,S_L$, where the sparsity on each index set $S_n$ is $k_n$, with $\sum_{n=1}^L k_n=k$.  The general idea of the analysis is to reduce the usual IHT theoretical guarantees on the full equation $Ax=b$, to $L$ smaller problems of the linear system restricted to each $S_n$.  
As IHT guarantees depend on the coherence of the sensing matrix $A$ and the sparsity of $x$, it is beneficial to consider these individual subproblems which each have smaller sparsity values. Of course, additional factors appear in this analysis, as the $L$ smaller systems are interconnected by the sensing matrix, and not truly independent.  However, to leading order, the results are analogous with the reduction of $k$ to $k_n$ on each set $S_n$.  

\begin{theorem}
\label{thm2}
Let $\{x^{(t)}\}$ be the sequence generated from the structured IHT algorithm given by \eqref{eq:siht} for the equation $Ax^*=b+\epsilon$, where $A$ has been column-normalized $(\|A_j\|=1$ for all $j$).  Let $S_1,\dots,S_L$ be mutually disjoint index sets with corresponding sparsity values $k_1,\dots k_L$ such that $\sum_{i=1}^L k_i=k$.  Then, for all $n$, if $x_{S_n}$ is $k_n$-sparse, one has the error bound 
\begin{align}
\label{eq:thm2}
\|x_{S_n}^{(t)}- x^*_{S_n} \|_1\le (3&\mu_{S_n} k_n)^t\|x_{S_n}^{(0)}-x^*_{S_n}\|_1 +\frac{3k_n}{1-3\mu_{S_n} k_n}\|[A^*\epsilon]_{S_n}\|_\infty \nonumber \\
+&\sum_{s=1}^{t} \left[{t\choose s}\tilde{\rho}^{s}\rho^{t-s} E_s(n)\ +\frac{\tilde{\rho}^{s}}{(1-\rho)^{s+1}} \|A^*\epsilon\|_\infty K_s(n) \right] \ ,
\end{align}
where we define \begin{align}
    &\tilde{\rho}=\max_{\substack{1\le m,n\le L\\m\neq n}} 3k_n\mu_{S_n,S_m} \qquad ; \qquad \rho=\max_{\substack{1\le n\le L}} 3k_n\mu_{S_n} \nonumber\\ 
    &E_s(n) = \sum_{m_1\neq n}\sum_{m_2\neq m_1}\cdots\sum_{m_s\neq m_{s-1}} \|x^{(0)}_{S_{m_s}}-x^*_{S_{m_s}}\|_1 \nonumber \\
    &K_s(n) = \sum_{m_1\neq n}\sum_{m_2\neq m_1}\cdots\sum_{m_s\neq m_{s-1}} 3k_{m_s} \ .
\end{align} 
\end{theorem}

We will prove Theorem \ref{thm2} by using the following two lemmas.  Lemma \ref{lemma1} is adapted from \cite{wang2015linear} whose proof can be found therein (Appendix D).  Immediately following the statement of Lemma \ref{lemma1}, we state and prove the second lemma, Lemma \ref{lemma2}.  

\begin{lemma}
\label{lemma1}
Let $\{x^{(t)}\}$ be the sequence generated from the structured IHT algorithm given by \eqref{eq:siht} for a column-normalized matrix $A$.  For any $t\ge1$, let $z^{(t)}=x^{(t)} + A^*(b - A x^{(t)} )$, which is the result of the structured IHT algorithm before any thresholding.  Then for any $t\ge1$ and any $1\le n\le L$,
\begin{equation}
\label{eq:11}
\|x^{(t+1)}_{S_n}-x^*_{S_n}\|_1\le 3k_n \|z^{(t+1)}_{S_n}-x^*_{S_n}\|_\infty \ . 
\end{equation}
\end{lemma}

\begin{lemma}
\label{lemma2}
With the same hypotheses as in Lemma \ref{lemma1}, for any $t\ge1$ and any $1\le n\le L$,
\begin{equation}
\|z^{(t+1)}_{S_n}-x^*_{S_n}\|_\infty \le \mu_{S_n}\|x^{(t)}_{S_n}-x^*_{S_n}\|_1+\sum_{m\neq n} \mu_{S_n,S_m}\|x^{(t)}_{S_m}-x^*_{S_m}\|_1+\|[A^*\epsilon]_{S_n}\|_\infty \ .
\end{equation}
\end{lemma}
\begin{proof}[Proof of Lemma \ref{lemma2}]
For any $t\ge 1$ and restriction to index set $S_n$,
\begin{equation}
z^{(t+1)}_{S_n}=x^{(t)}_{S_n}-\left[A^*(Ax^{(t)}-b)\right]_{S_n} \ .
\end{equation}
Substituting in $b=Ax^*+\epsilon$ and subtracting $x^*_{S_n}$ from both sides we obtain
\begin{equation}
z^{(t+1)}_{S_n}-x^*_{S_n}=x^{(t)}_{S_n}-x^*_{S_n}-\left[A^*A(x^{(t)}-x^*)\right]_{S_n}+[A^*\epsilon]_{S_n} \ .
\end{equation}
As $x^{(t)}_{S_n}-x^*_{S_n}=[x^{(t)}-x^*]_{S_n}$, we can rewrite this as
\begin{equation}
z^{(t+1)}_{S_n}-x^*_{S_n}=\left[\left(I-A^*A\right)(x^{(t)}-x^*)\right]_{S_n}+[A^*\epsilon]_{S_n} \ .
\end{equation}
By taking infinity norms of both sides and applying the triangle inequality, we have
\begin{equation}
\label{eq:prop_m1}
\|z^{(t+1)}_{S_n}-x^*_{S_n}\|_\infty\le \left\|\left[\left(I-A^*A\right)(x^{(t)}-x^*)\right]_{S_n}\right\|_\infty+\|[A^*\epsilon]_{S_n}\|_\infty \ .
\end{equation}
We will now show that the following inequality holds:
\begin{equation}
\label{eq:prop1}
\left\|\left[\left(I-A^*A\right)(x^{(t)}-x^*)\right]_{S_n}\right\|_\infty \le \mu_{S_n}\|x^{(t)}_{S_n}-x^*_{S_n}\|_1+\sum_{m\neq n} \mu_{S_n,S_m}\|x^{(t)}_{S_m}-x^*_{S_m}\|_1 \ .
\end{equation}
As $A$ is column-normalized, $I-A^*A$ has zeros on the diagonal, which implies that for $i\in S_m$ and $j\in S_n$ we have
\begin{equation}
|(I-A^*A)_{ij}|\le \mu_{S_m,S_n} \ .
\end{equation}
Thus, we have
\begin{align*}
\left\|\left[\left(I-A^*A\right)(x^{(t)}-x^*)\right]_{S_n}\right\|_\infty& =  \max_{j\in S_n}\left| \sum_{i=1}^N (I-A^*A)_{ji}(x^{(t)}-x^*)_i \right| \\
\le  \max_{j\in S_n}\sum_{i\in S_n} & \left| (I-A^*A)_{ji}(x^{(t)}-x^*)_i \right|+\sum_{i\not\in S_n} \left| (I-A^*A)_{ji}(x^{(t)}-x^*)_i \right| \\
\le & \mu_{S_n}\sum_{i\in S_n} \left|(x^{(t)}-x^*)_i \right|+\sum_{S_m\neq S_n}\mu_{S_n,S_m}\sum_{i\in S_m}  \left|(x^{(t)}-x^*)_i \right|  \\
= & \mu_{S_n}\|x^{(t)}_{S_n}-x^*_{S_n}\|_1+\sum_{m\neq n} \mu_{S_n,S_m}\|x^{(t)}_{S_m}-x^*_{S_m}\|_1 \ ,
\end{align*}
which proves \eqref{eq:prop1}.
By applying \eqref{eq:prop1} to \eqref{eq:prop_m1}, we arrive at the desired conclusion.  
\end{proof}
With these two lemmas in hand, we can now prove Theorem \ref{thm2}.  
\begin{proof}[Proof of Theorem \ref{thm2}]
By combining Lemmas \ref{lemma1} and \ref{lemma2}, we obtain the inequality
\begin{equation}
\label{eq:1}
\|x^{(t)}_{S_n}-x^*_{S_n}\|_1\le 3\mu_{S_n}k_n\|x^{(t-1)}_{S_n}-x^*_{S_n}\|_1 +3k_n\sum_{m\neq n}\mu_{S_n,S_m}\|x^{(t-1)}_{S_m}-x^*_{S_m}\|_1+3k_n\|[A^*\epsilon]_{S_n}\|_\infty .
\end{equation}
Iterating back once in the term  $\|x^{(t-1)}_{S_n}-x^*_{S_n}\|_1 $ using \eqref{eq:1}  yields 
\begin{align}
\|x^{(t)}_{S_n}-x^*_{S_n}\|_1\le (3\mu_{S_n}k_n)^2\|x^{(t-2)}_{S_n}-x^*_{S_n}\|_1+(3k_n)^2\mu_{S_n}\sum_{m\neq n}\mu_{S_n,S_m}\|x^{(t-2)}_{S_m}-x^*_{S_m}\|_1\nonumber \\ +3k_n\sum_{m\neq n}\mu_{S_n,S_m}\|x^{(t-1)}_{S_m}-x^*_{S_m}\|_1+3k_n(1+3\mu_{S_n}k_n)\|[A^*\epsilon]_{S_n}\|_\infty .
\end{align}
By continuing to iterate back to $t=0$ in only the first error term on the right hand side (which is restricted to $S_n$), we obtain
\begin{align}
\label{eq:2}
&\|x^{(t)}_{S_n}-x^*_{S_n}\|_1\le (3\mu_{S_n}k_n)^t\|x^{(0)}_{S_n}-x^*_{S_n}\|_1 +\nonumber\\
&\qquad \quad\sum_{\ell=0}^{t-1}(3\mu_{S_n}k_n)^\ell\sum_{m\neq n} 3k_n\mu_{S_n,S_m}\|x^{(t-1-\ell)}_{S_m}-x^*_{S_m}\|_1+3k_n\sum_{\ell=0}^{t-1}(3\mu_{S_n}k_n)^\ell\|[A^*\epsilon]_{S_n}\|_\infty .
\end{align}
We now iterate back similarly in each of the remaining error terms $\|x^{(t-1-\ell)}_{S_m}-x^*_{S_m}\|_1$.  However, we will use upper bounds \begin{equation}
\label{eq:ub}
    \tilde{\rho}=\max_{\substack{1\le m,n\le L\\m\neq n}} 3k_n\mu_{S_n,S_m} \quad \text{     and     } \quad \rho=\max_{1\le n\le L}3k_n\mu_{S_n}\ ,
\end{equation} to arrive at a more concise expression.  By recursively applying \eqref{eq:2} to its middle term on the right hand side, and applying the bounds in \eqref{eq:ub} (and replacing $m$ with $m_1$), one obtains
\begin{align}
\label{eq:new}
    \sum_{\ell=0}^{t-1}(3\mu_{S_n}k_n)^\ell\sum_{{m_1}\neq n} 3k_n&\mu_{S_n,S_{m_1}}\|x^{(t-1-\ell)}_{S_{m_1}}-x^*_{S_{m_1}}\|_1 \le  \sum_{\ell=0}^{t-1}\rho^\ell\sum_{{m_1}\neq n} \tilde{\rho} \rho^{t-1-\ell}\|x^{(0)}_{S_{m_1}}-x^*_{S_{m_1}}\|_1 \nonumber\\
    +& \sum_{\ell=0}^{t-1}\rho^\ell\sum_{{m_1}\neq n} \tilde{\rho}\sum_{\ell_1=0}^{t-\ell-2}\rho^{\ell_1}\sum_{m_2\neq m_1}\tilde{\rho}\|x^{(t-2-\ell-\ell_1)}_{S_{m_2}}-x^*_{S_{m_2}}\|_1 \\
     +& \sum_{\ell=0}^{t-1}\rho^\ell\sum_{m_1\neq n} 3k_{m_1}\tilde{\rho}\sum_{\ell_1=0}^{t-\ell-2}\rho^{\ell_1}\|[A^*\epsilon]_{S_{m_1}}\|_\infty \label{eq:newerr}.
\end{align}
This first term can be simplified as
\begin{align}
    \sum_{\ell=0}^{t-1}\rho^\ell\sum_{m_1\neq n} \tilde{\rho} \rho^{t-1-\ell}\|x^{(0)}_{S_{m_1}}-x^*_{S_{m_1}}\|_1 = & t\tilde{\rho}\rho^{t-1}\sum_{{m_1}\neq n}  \|x^{(0)}_{S_{m_1}}-x^*_{S_{m_1}}\|_1 \ .
\end{align}
Repeating this process to recursively substitute \eqref{eq:2} in for the middle term in \eqref{eq:new} yields terms of the form
\begin{align}
\label{eq:recurse_err}
  \sum_{\ell=0}^{t-1}\rho^\ell\sum_{m_1\neq n} \tilde{\rho}\sum_{\ell_1=0}^{t-\ell-2}\rho^{\ell_1}\sum_{m_2\neq m_1}&\tilde{\rho}\cdots\hspace{-5ex}\sum_{\ell_s=0}^{t-\ell-\cdots-\ell_{s-1}-s)}\hspace{-5ex}\rho^{\ell_s}\sum_{m_s\neq m_{s-1}}\tilde{\rho}\rho^{t-\ell-\cdots-\ell_{s-1}-s} \|x^{(0)}_{S_{m_s}}-x^*_{S_{m_s}}\|_1\nonumber\\
  = &{t\choose s}\tilde{\rho}^{s}\rho^{t-s} \sum_{m_1\neq n}\sum_{m_2\neq m_1}\cdots\sum_{m_s\neq m_{s-1}} \|x^{(0)}_{S_{m_s}}-x^*_{S_{m_s}}\|_1 \nonumber \\
  = &{t\choose s}\tilde{\rho}^{s}\rho^{t-s} E_s(n) \ .
\end{align}
The only terms left to track are the error terms of the form \eqref{eq:newerr}.  Similarly, each of these terms takes the form
\begin{align}
\label{eq:recurse_eps}
    \sum_{\ell=0}^{t-1}\rho^\ell\sum_{m\neq n} \tilde{\rho}&\sum_{\ell_1=0}^{t-\ell-2}\rho^{\ell_1}\sum_{m_1\neq m}\tilde{\rho}\cdots\hspace{-5ex}\sum_{\ell_s=0}^{t-\ell-\cdots-\ell_{s-1}-s)}\hspace{-5ex}\rho^{\ell_s}\sum_{m_s\neq m_{s-1}}3k_{m_s} \|[A^*\epsilon]_{S_{m_s}}\|_\infty\nonumber \\
    \le &\|A^*\epsilon\|_\infty\frac{\tilde{\rho}^{s}}{(1-\rho)^{s+1}} \sum_{m_1\neq n}\sum_{m_2\neq m_1}\cdots\sum_{m_s\neq m_{s-1}} 3k_{m_s} \nonumber \\
    = & \frac{\tilde{\rho}^{s}}{(1-\rho)^{s+1}} \|A^*\epsilon\|_\infty K_s(n) \ ,
\end{align}
where each geometric series in $\rho$ has been bounded by its limit of $(1-\rho)^{-1}$ as $t$ goes to infinity.  For fixed $t$, the recursive terms \eqref{eq:recurse_err} and \eqref{eq:recurse_eps} range from $s=1$ to $s=t$. 
Combining this together yields
\begin{align}
    &\|x^{(t)}_{S_n}-x^*_{S_n}\|_1\le (3\mu_{S_n}k_n)^t\|x^{(0)}_{S_n}-x^*_{S_n}\|_1 +3k_n\sum_{\ell=0}^{t-1}(3\mu_{S_n}k_n)^\ell\|[A^*\epsilon]_{S_n}\|_\infty\nonumber\\
&\qquad \quad +\sum_{s=1}^{t} \left[{t\choose s}\tilde{\rho}^{s}\rho^{t-s} E_s(n) +\frac{\tilde{\rho}^{s}}{(1-\rho)^{s+1}} \|A^*\epsilon\|_\infty K_s(n)\right] \  .
\end{align}
To obtain the final version of inequality \eqref{eq:thm2}, we simply bound the error term containing $\epsilon$ on $S_n$ with its limit as $t$ goes to $\infty$.
\end{proof}

Theorem \ref{thm2} gives an error bound for each subvector of $x$, and shows that it converges linearly (to leading order) to the true solution within an error factor.  Note that while not a requirement for the theorem, the condition $\tilde{\rho}<\rho$ is necessary to consider the first two terms in \eqref{eq:thm2} the leading order terms.  This condition holds for many applications.   We state and prove the following corollary which provides an error estimate for the entire vector $x$, not only for each subvector $x_{S_n}$.

\begin{corollary}
\label{cor1}
Using the same definitions and assumptions as in Theorem \ref{thm2}, as well as the condition $\rho+(L-1)\tilde{\rho}<1$, then for all $t\ge 1$, 
\begin{align}
\label{eq:cor}
\|&x^{(t)}-x^*\|_1 \le \left(\rho+ (L-1)\tilde{\rho}\right)^t\|x^{(0)}-x^*\|_1  + \frac{3k\|A^*\epsilon\|_\infty}{1-\rho-(L-1)\tilde{\rho}}\ .
\end{align}
\end{corollary}
\begin{proof}
As the $S_n$ are disjoint and all values for which $x^*\neq0$ live in some $S_n$, 
\begin{equation}
\label{eq:simple}
\|x^{(t)}-x^*\|_1 = \sum_{n=1}^N \|x_{S_n}^{(t)}-x^*_{S_n}\|_1 \ .
\end{equation}
By applying this equality to the conclusion \eqref{eq:thm2} of Theorem \ref{thm2}, we have
\begin{align}
\label{eq:cor1}
\|x^{(t)}-x^*\|_1 \le  &\sum_{n=1}^L (3\mu_{S_n} k_n)^t\|x_{S_n}^{(0)}-x^*_{S_n}\|_1 + \sum_{n=1}^L\frac{3k_n}{1-3\mu_{S_n} k_n}\|[A^*\epsilon]_{S_n}\|_\infty \nonumber\\
&+\sum_{n=1}^L\sum_{s=1}^{t}{t\choose s}\tilde{\rho}^{s}\rho^{t-s} E_s(n)
+ \sum_{n=1}^L\sum_{s=1}^{t}\frac{\tilde{\rho}^{s}}{(1-\rho)^{s+1}} K_s(n)\|A^*\epsilon\|_\infty \ .
\end{align}
The terms on the first line of \eqref{eq:cor1} can be bounded by 
\begin{align}\label{eq:cor_bound1}
    \sum_{n=1}^L (3&\mu_{S_n} k_n)^t\|x_{S_n}^{(0)}-x^*_{S_n}\|_1 + \sum_{n=1}^L\frac{3k_n}{1-3\mu_{S_n} k_n}\|[A^*\epsilon]_{S_n}\|_\infty \le \nonumber\\ &\rho^t \sum_{n=1}^L\|x_{S_n}^{(0)}-x^*_{S_n}\|_1 + \frac{\|A^*\epsilon\|_\infty}{1-\rho}\sum_{n=1}^L 3 k_n = \rho^t \|x^{(t)}-x^*\|_1 +\frac{3k\|A^*\epsilon\|_\infty}{1-\rho} \ .
\end{align}
Turning our attention towards the two terms on the second line of \eqref{eq:cor1}, by a counting argument, the sums of $K_s(n)$ and $E_s(n)$ can be computed as 
\begin{align}\label{eq:Kns_def}
   & \sum_{n=1}^L E_s(n) = \sum_{n=1}^L \sum_{m_1\neq n}\sum_{m_2\neq m}\cdots\sum_{m_s\neq m_{s-1}} \|x^{(0)}_{S_{m_s}}-x^*_{S_{m_s}}\|_1=(L-1)^{s}\|x^{(0)}-x^*\|_1 \nonumber \\
   & \sum_{n=1}^L K_s(n) = \sum_{n=1}^L\sum_{m_1\neq n}\sum_{m_2\neq m}\cdots\sum_{m_s\neq m_{s-1}} 3k_{m_s}\|A^*\epsilon\|_\infty = (L-1)^{s}3k\|A^*\epsilon\|_\infty \ . 
\end{align}
Substituting \eqref{eq:cor_bound1} and \eqref{eq:Kns_def} into \eqref{eq:cor1} yields
\begin{align}
\label{eq:cor11}
\|x^{(t)}-x^*\|_1 \le  \sum_{s=0}^{t}{t\choose s}\tilde{\rho}^{s}\rho^{t-s} (L-1)^s\|x^{(0)}-x^*\|_1\  +\frac{3k\|A^*\epsilon\|_\infty}{1-\rho}\sum_{s=0}^{t}\left(\frac{(L-1)\tilde{\rho}}{1-\rho} \right)^s \ ,
\end{align}
where the terms from \eqref{eq:cor_bound1} have been included in the summations, which now start from $s=0$.  The first summation can be rewritten as $(\rho+(L-1)\tilde{\rho})^t$ by the binomial theorem.  Under the assumptions that $\rho+(L-1)\tilde{\rho}<1$, the remaining geometric series converges.  Bounding this series by its limit yields the desired form in  \eqref{eq:cor}.  
\end{proof}
Hence, similar to IHT from Theorem \ref{thm1}, we have linear convergence.  The rate of convergence for structured IHT is $\rho+(L-1)\tilde{\rho}$.  If $\tilde{\rho}=\rho$, this rate reduces to $L\rho=3k\mu$, which is the rate of convergence for IHT.  Whenever $\tilde{\rho}<\rho$ (which is true in many applications), the provided theory expects faster convergence for structured IHT.  This also yields a stronger guarantee of convergence, as $\rho+(L-1)\tilde{\rho}$ can be less than 1 even if $3\mu k$ is not.  
In the following section, we will see that in a typical scenario, the theory of structured IHT as given by Corollary \ref{cor1} can give significantly stronger theoretical guarantees than the IHT theory from Theorem \ref{thm1}. 
Note that in the case when we have only one index set $S_1$ which includes all indices of $x$, \eqref{eq:cor} reduces to the result of Theorem \ref{thm1} by setting $L=1$.

We remark that the result of Corollary \ref{cor1} can be obtained in a more direct manner from immediately bounding all relevant terms in \eqref{eq:1} by $\rho$ and $\tilde{\rho}$ and subsequently summing both sides from 1 to $L$.  However, we have chosen to state Theorem \ref{thm2} which gives additional information (to leading order) for each index set.  In particular, index sets with smaller values of $\mu_{S_n}k_n$ will converge faster.  Moreover, if $\|[A^*\epsilon]_{S_n}\|_\infty$ is smaller, there will be less noise introduced to the reconstruction of $x_{S_n}^*$).  

Lastly, it is important to note that, while the theory points to a benefit of using structured IHT, it may be prohibitive to obtain the structured sparsity information needed to apply the theory of structured IHT.  Much of compressed sensing theory, in general, requires the overall sparsity level to be known.  Even when this level is unknown, there are methods for estimating the true sparsity level \cite{Lopes13,7506083}.  While it is more difficult to estimate the individual sparsity level for each index set, there are still methods for determining the structure of the sparsity \cite{shervashidze2015learning,wen2016learning}.  In the following numerical simulations, we will look at some heuristic methods for estimating the sparsity level for each index set, but note that this can be more difficult in other applications.

\section{Numerical Simulations}
\label{sec:ongrid}
\subsection{Inverse Source Problem Setup}
To demonstrate the practicality and performance of the structured IHT algorithm, we apply it to an inverse source problem (ISP) where the source is known to be a superposition of a small number of plane waves.  This problem shares strong similarities with direction-of-arrival (DOA) and angle-of-arrival (AOA) type problems in the literature, for which many efficient and accurate algorithms exist \cite{shan1985spatial,stoica1990maximum,
tuncer2009classical,yang2018sparse,yang2012off}.  However, in the considered ISP, the amplitudes of the plane waves are also unknown, as opposed to only solving for the incident angles.  For simplicity, we also consider detector geometries that are unlikely in a typical DOA or AOA setting.

Consider $k$ unknown sources that are emitted simultaneously.  We assume each source $s_j$ is a plane wave of the form
\begin{equation}
s_j(\mathbf{r})=a_j e^{i\omega\mathbf{r}\cdot\Theta_j} \ .
\end{equation}
Each source is thus completely characterized by its  frequency $\omega$, amplitude $a_j$, and incident angle $\Theta_j$.  We assume that the frequency is known and constant for all sources, with the only unknowns being the amplitudes and incident angles.  In the following simulations, the true amplitudes $a_j$ will be real, but this will not be assumed to be true {\it a priori}.   The field $u$ arriving at any point $\mathbf{r}\in\mathbb{R}^3$ is given by the superposition of these simultaneously emitted sources by
\begin{equation}
\label{eq:isp_1}
u(\mathbf{r})=\sum_{j=1}^k a_j e^{i\omega\mathbf{r}\cdot\Theta_j} \ .
\end{equation}
By making measurements of the field $u$ at known detector locations $\mathbf{d}_1,\dots\mathbf{d}_M$, the goal is to recover the values $a_j$ and $\theta_j$ for all $1\le j \le k$.   In the simplest version of this ISP, the value $k$ is known.  We will consider the problem when $k$ is unknown in Section \ref{sec:offgrid}.  

To apply the structured IHT algorithm to this ISP, the problem needs to be first formulated as a sparse linear recovery problem.  One could equally try one of many other sparsity-promoting algorithms for solving this problem, but in this paper, it will be used as an example to demonstrate the practicality of structured IHT.  We remark that there are also approaches in the DOA and AOA literature that do not rely on sparsity \cite{shan1985spatial,stoica1990maximum,
tuncer2009classical}.    

To write this as a linear system, the measurements at the detector locations $\mathbf{d}_m$ in \eqref{eq:isp_1} can be written as the $M\times k$ linear system
\begin{equation}
\label{eq:isp_2}
\begin{bmatrix}
e^{i\omega \mathbf{d}_1\Theta_1} &\cdots&  e^{i\omega \mathbf{d}_1\Theta_k} \\
\vdots & & \vdots \\
e^{i\omega \mathbf{d}_M\Theta_1} &\cdots&  e^{i\omega \mathbf{d}_M\Theta_k} 
\end{bmatrix}
\begin{bmatrix}
a_1 \\ \vdots \\ a_k
\end{bmatrix}=
\begin{bmatrix}
u(\mathbf{d}_1)\\ \vdots \\ u(\mathbf{d}_M)
\end{bmatrix} \ .
\end{equation}
However, as the true incident angles $\Theta_j$ are unknown, the matrix in \eqref{eq:isp_2} is also not known.  In the on-grid scenario (the off-grid case is considered in Section \ref{sec:offgrid}), we assume each $\Theta_j$ is equal to one of $N$ test angles, $\Phi_1,\dots,\Phi_N$, where typically $N\gg M> k$.  We now rewrite \eqref{eq:isp_2} in the form 
\begin{equation}
\label{eq:isp_3}
\begin{bmatrix}
e^{i\omega \mathbf{d}_1\Phi_1} &e^{i\omega \mathbf{d}_1\Phi_2}&\cdots&  e^{i\omega \mathbf{d}_1\Phi_N} \\
\vdots && & \vdots \\
e^{i\omega \mathbf{d}_M\Phi_1} &e^{i\omega \mathbf{d}_M\Phi_2}&\cdots&  e^{i\omega \mathbf{d}_M\Phi_N}
\end{bmatrix}
\begin{bmatrix}
x_1 \\ \vdots\\\vdots \\ x_N
\end{bmatrix}=
\begin{bmatrix}
u(\mathbf{d}_1)\\ \vdots \\ u(\mathbf{d}_M)
\end{bmatrix} \ .
\end{equation}
This now underdetermined system contains the same information as \eqref{eq:isp_2}, but takes into account that the true values of $\Theta_j$ are unknown.  However, as there are only $k$ values of $\Theta$, there should only be $k$ nonzero values of $x_j$.  If $\Theta_j=\Phi_{j'}$ for some $j$ and $j'$, then $x_{j'}=a_j$.  If $\Phi_{j'}$ is not equal to any $\Theta_j$, then $x_{j'}=0$.  As $M<N$, this is a standard underdetermined linear sparse recovery problem.    We denote this linear system \eqref{eq:isp_3} by $Ax=b$, where $A$ is an $M\times N$ matrix and $x$ is an $N\times 1$ vector that is $k$-sparse.

\subsection{Comparison to Theory}
\label{sec:comptotheory}
The theory of solving this sparse linear system by structured IHT centers around the coherence of $A$.  We can estimate the coherence for a fixed geometry of detector locations.  For the remainder of the paper, the detector locations $\mathbf{d}_m$ will be assumed to be uniformly spaced on a sphere centered at the origin.

Consider two columns $A_j$ and $A_\ell$ from \eqref{eq:isp_3}.  Computing the coherence by \eqref{eq:mu_def}, we have
\begin{align}
\label{eq:coh1}
\frac{\langle A_\ell, A_j\rangle}{\|A_\ell\|\|A_j\|} =\frac{1}{M} \sum_{m=1}^M e^{i\omega\mathbf{d}_m\Phi_\ell} e^{-i\omega\mathbf{d}_m\Phi_j}  \ .
\end{align}
Assuming that the detector locations $\mathbf{d}_m$ are uniformly spaced on a sphere of radius $R$ centered at the origin, for sufficiently large values of $M$ the coherence \eqref{eq:coh1} is approximated by
\begin{equation}
\frac{\langle A_\ell, A_j\rangle}{\|A_\ell\|\|A_j\|}\approx \frac{1}{4\pi}\int_{S^2} e^{i\omega R \mathbf{\hat{d}}\cdot (\Phi_\ell-\Phi_j)} \ d\mathbf{\hat{d}} = \frac{\sin(\omega R|\Phi_\ell-\Phi_j|)}{\omega R|\Phi_\ell-\Phi_j|} \ ,
\end{equation} 
where the integral is taken over all unit vectors $\mathbf{\hat{d}}$ on the unit sphere $S^2$.  Letting $h$ be the minimum distance between any $\Phi_\ell\in S$ and $\Phi_j\in S'$, the restricted coherence of $A$ can be approximated by
\begin{equation}
\label{eq:coh_A}
\mu_{S,S'}(A)=\max_{\substack{\ell\in S, j\in S'\\\ell\neq j}}\frac{\langle A_\ell, A_j\rangle}{\|A_\ell\|\|A_j\|} \approx \frac{\sin(\omega hR)}{\omega hR} \ . 
\end{equation}
This approximation is accurate for small values of $\omega h R$, but becomes less accurate as  $\omega h R$ increases.  Note that $\omega R$ is a fixed parameter that depends on the detector placements and the known frequency.  In general, an ISP with a  larger value of $\omega R$ tends to have greater capacity for resolution due to the higher frequency.  This is accompanied by a smaller coherence value.  Equation \eqref{eq:coh_A} provides some justification for the restriction that $(L-1)\tilde{\rho}<1-\rho$ in the statement and proof of Corollary \ref{cor1}. As will be seen in the numerical simulations, it is natural for each index set $S$ to include incident angles that are in close proximity with one another.  Thus, the minimum distance $h$ between two source locations within a set $S$ will often be significantly smaller than the minimum distance $h'$ between one angle in $S$ and another in $S'$.  Thus, according to \eqref{eq:coh_A}, as long as $\omega R$ is not too large, and $L$ is not too large, this condition will hold.

We first conducted a simple numerical experiment to compare the performance of the algorithm with these coherence approximations within the context of the theory presented in Section \ref{sec:analysis}.
Consider the case of $k=3$ plane wave sources, all with unit amplitude $a_j=1$.  We write any incident angle using spherical coordinates as $ \Theta_j=(\cos\theta_j\sin\phi_j,\sin\theta_j\sin\phi_j,\cos\phi_j)$, where $\theta_j\in[0,2\pi)$ and $\phi_j\in[0,\pi]$.  Consider a simple problem where the incident angle of any source is known to have fixed azimuthal angle  of $\phi=\pi/2$.   We thus create a one dimensional grid in $\theta$, with $N$ uniformly spaced values between 0 and $2\pi$.  This gives the candidate directions $\Phi_n$  as 
\begin{equation}
\label{eq:1d_angles}
\Phi_n = (\cos (2\pi n/N), \sin (2\pi n/N), 0) \ .
\end{equation}
The value of $N$ was chosen to be 200, with the $k=3$ nonzero values (with unit amplitude) placed at $n=104, 106$, and 165.  
The number of detectors chosen was $M=100$, which makes this problem underdetermined.  The detectors were placed uniformly about a sphere, with the dimensionless parameter set $\omega R=275$.  The placement of the detectors was performed according to the Fibonacci spiral method \cite{hardin2016comparison}.  No noise was added to the data.  

Two different sparse recovery algorithms were run on this setup to recover the 3 plane wave sources.  The first algorithm was the IHT algorithm, with thresholding performed on the $k=3$ largest entries.  This was then compared with the structured IHT algorithm.  The chosen structure consisted of knowing that there were $k_1=2$ nonzero entries on the range of $n$ values with $98\le n \le 112$ and $k_2=1$ nonzero entries on the range of $n$ values $159 \le n \le 171$.  These two sets of indices made up $S_1$ and $S_2$ respectively.  We remark that this structure was chosen arbitrarily for this toy problem, but has some plausible basis in the fact that there are two sources close together (contained in $S_1$), with a single isolated source that should be easier to recover (contained in $S_2$).
 
 With these chosen parameters, the coherence of $A$ is estimated by \eqref{eq:coh_A} as $0.081879$ which was numerically computed to be $0.081885$ (both displayed to 5 significant digits).  As expected, the coherence values when restricted to either $S_1$ or $S_2$ are nearly identical (0.081881 and 0.081884), as the value of $h=0.031$ in \eqref{eq:coh_A} is unchanged.  The coherence $\mu_{S_1,S_2}(A)$ is roughly four times smaller at 0.019788.  This value does not tightly follow the approximation given by \eqref{eq:coh_A} as the two closest angles between $S_1$ and $S_2$ are $\Phi_{112}$ and $\Phi_{159}$ which are separated by $h=1.32$.   In this case, $\omega h R$ is too large for \eqref{eq:coh_A} to be accurate.  A visual for these coherence values is shown below in Fig.~\ref{fig:theory_coh}.  As $L=2$, the convergence condition $3(0.081884)(2) + (2-1)(3)0.019788(1) =0.550668<1$ holds.
 \begin{figure}[h!]
\centering
\includegraphics[width=0.59\textwidth]{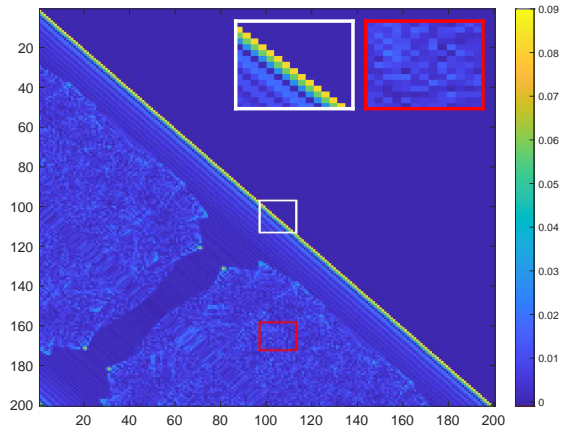}
\caption{ \label{fig:theory_coh} Visual of the values $\langle A_j,A_\ell\rangle/(\|A_j\|\|A_\ell\|)$ for the matrix $A$ given by \eqref{eq:isp_3} with $\omega R=275$.   Only the values for $j<\ell$ are plotted as the terms are symmetric.  For the index sets $S_1=\{98,\dots,112\}$ and $S_2=\{159,\dots,171\}$, the restricted coherences $\mu_{S_1}$ and $\mu_{S_1,S_2}$ are restricted to the regions in the white and red boxes, respectively.  Larger magnifications of these regions  are shown in the upper right hand corner.  For this setup, $\mu=0.081885$, $\mu_{S_1}=0.081881$, and $\mu_{S_1,S_2}=0.019788$.
} 
\end{figure}

The $\ell_1$ norm error of the three algorithms are plotted in Fig.~\ref{fig:theory_err}.  The theoretical guarantees for the two algorithms from \eqref{eq:result_iht} and \eqref{eq:cor} are plotted in dashed lines.   The vast outperformance of the algorithm compared to theory is expected, as coherence bounds reflect worst possible cases.  In practice, we expect to beat these bounds, hopefully by a significant margin.  The entire experiment was conducted again after setting the dimensionless parameter to be smaller at $\omega R=88$, with similar results.  The notable exception in this case is that the theory for IHT does not guarantee convergence whereas the structured IHT theory does.  This is because $\mu\approx\rho=0.133$ while $\tilde{\rho}=0.0076$.  For IHT, $3\mu k = 3(0.133)(3)= 1.197 >1$, whereas for structured IHT, $\rho+(L-1)\tilde{\rho}=0.821<1$. These results are also plotted in Fig.~\ref{fig:theory_err}. 

\begin{figure}[h!]
\centering
\begin{subfigure}[b]{0.49\textwidth}
                \centering
                \includegraphics[width=\textwidth]{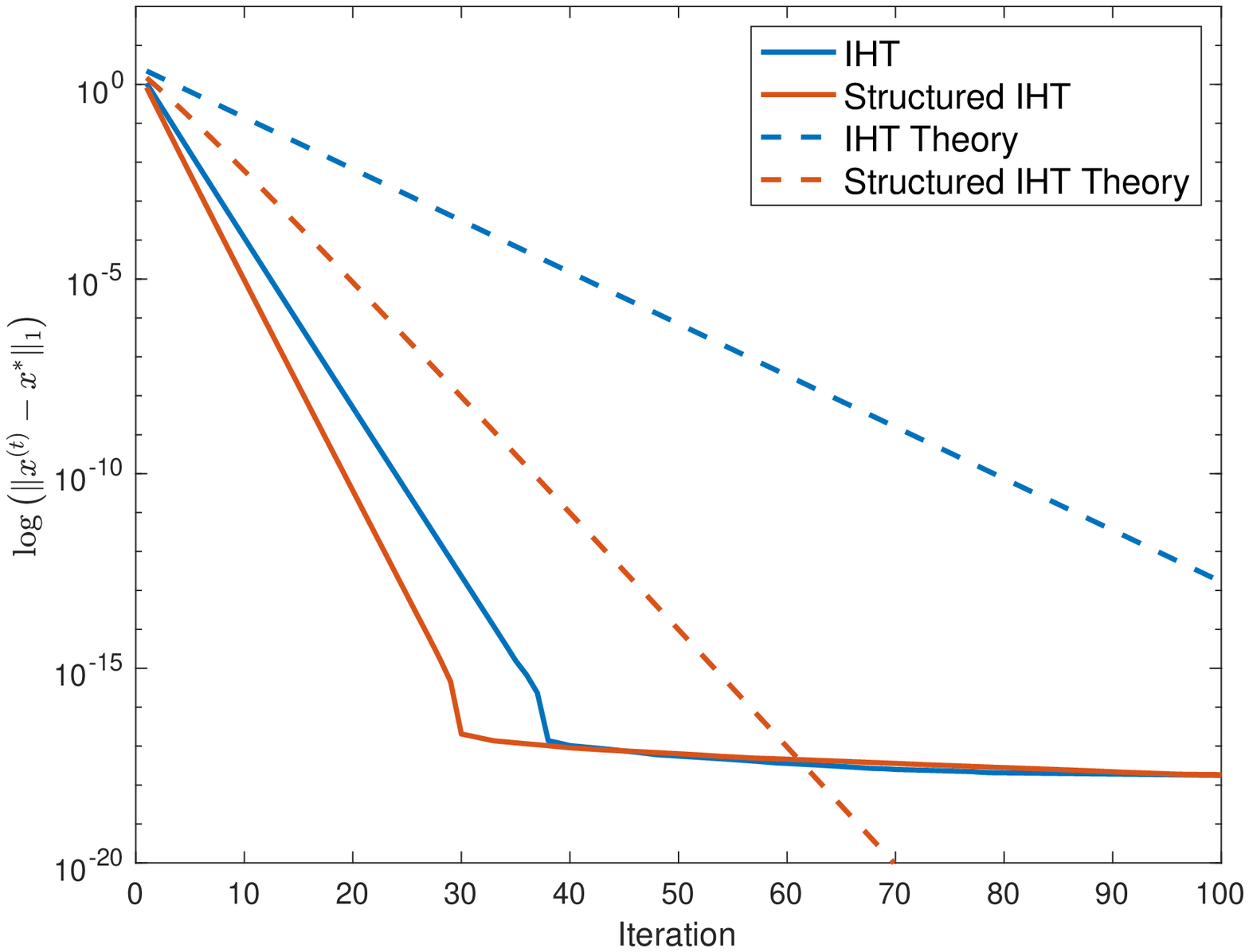}
                \caption{$\omega R=275$}
        \end{subfigure}
\begin{subfigure}[b]{0.49\textwidth}
                \centering
                \includegraphics[width=\textwidth]{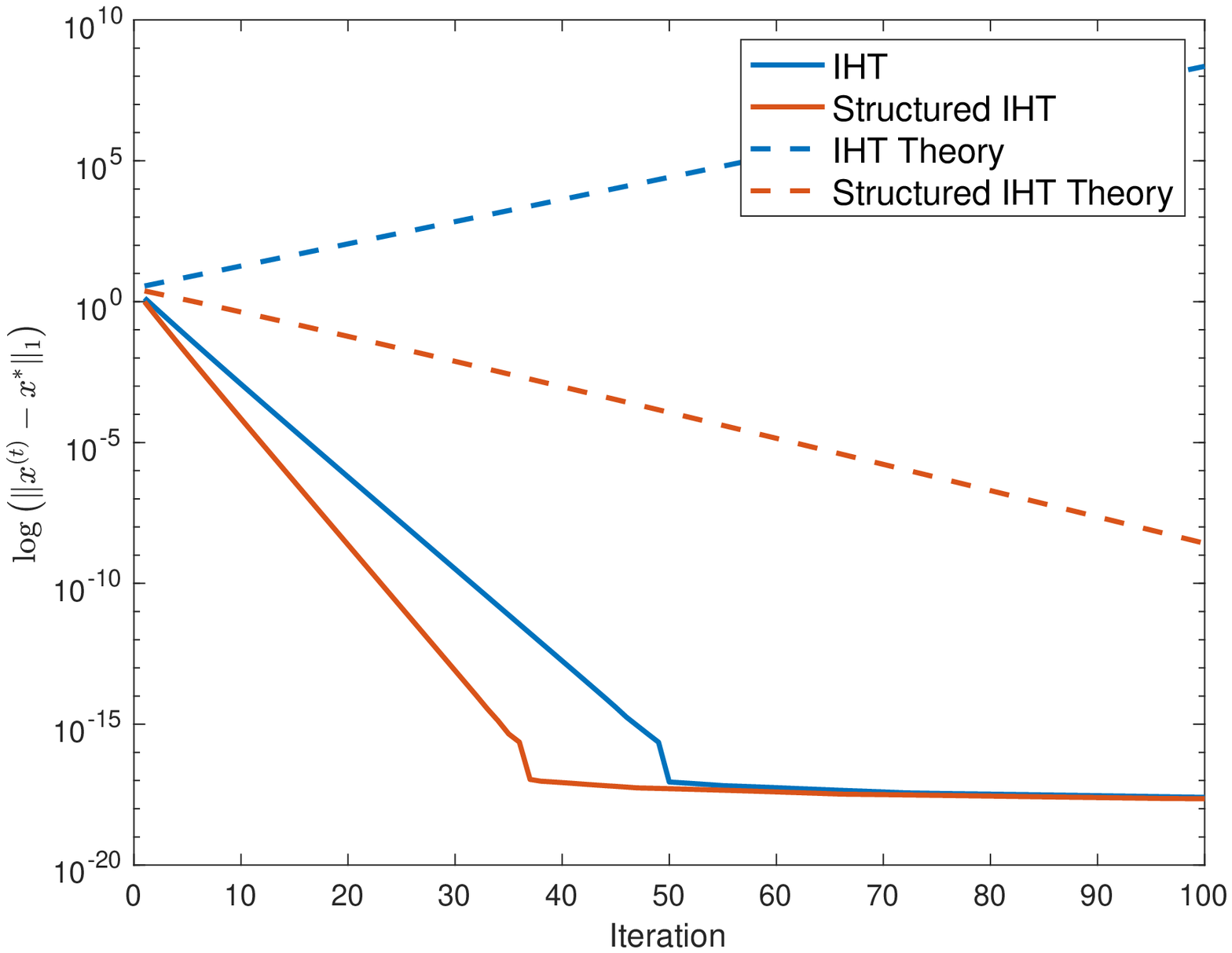}
                \caption{$\omega R=88$}
        \end{subfigure}
\caption{ \label{fig:theory_err} Comparison of the error $\|x^{(t)}-x^*\|_1$ for varying numerical simulations and theory. The left plot has parameter $\omega R=275$ with the right plot having $\omega R=88$. The blue solid line represents the numerical simulation of running the IHT algorithm.  The red orange solid line is for the structured IHT algorithm.  The corresponding dashed lines represent the theoretical guarantees for IHT, given by \eqref{eq:result_iht}, and structured IHT, given by \eqref{eq:cor}.   Note that for the plot on the right, structured IHT is guaranteed to converge while IHT is not.  However, both algorithms converged in practice.  }

\end{figure}

\subsection{On-Grid Numerical Simulations}
Building off of the numerical results of the simple experiment in Section \ref{sec:comptotheory}, we conduct numerical simulations in regimes beyond the presented theory.  Even though we expect structured IHT algorithm to outperform its theory, one main question to be answered is will it significantly outperform IHT without additional structure information?  And if so, how can one choose the structured index sets in practice on which we conduct the structured IHT algorithm? 

Our next numerical simulation addresses this first question on how structured IHT compares to IHT in practice.  We repeat the numerical experiment from Section \ref{sec:comptotheory} with fixed azimuthal angle $\phi=\pi/2$.  The number of candidates sources was increased to $N=1000$ in \eqref{eq:1d_angles}, as well as an increase in the number of detectors to $M=400$, which still left the problem significantly underdetermined.  The dimensionless parameter remained at $\omega R=275$.  This $400\times 1000$ sensing matrix $A$ then acted on a $k$-sparse vector $x$.  The $k$ nonzero entries of $x$ were chosen uniformly random, with each nonzero entry given unit amplitude.    

Five different algorithms were run on this setup: the standard IHT algorithm, and 4 versions of the structured IHT algorithm with varying structures.  For the structured IHT algorithm, the underlying vector was split into $L$ index sets of uniform size.  The index sets were given by
\begin{equation}
    S_j = \left\{\frac{(j-1)N}{L}+1,\dots,\frac{jN}{L} \right\} \qquad ,\qquad 1\le j \le L\ .
\end{equation}
The structured IHT algorithm was run separately for $L=2, 5, 10$, and $20$ divisions.  For each chosen structure, the correct sparsity level for $k_j$ for each index set $S_j$ was computed by inspection of the randomly generated vector $x$.  Note that this setup is already moving away from the optimal theory. In this setup, $\mu(A)\approx\mu_{S_j}(A)\approx\mu_{S_j,S_{j+1}}(A)$ for all $j$ as the index sets $S_j$ and $S_{j+1}$ contain neighboring sources of minimum spacing.  However, for $L>3$, one can find a $j$ and $j'$ such that $\mu_{S_j,S_{j'}}<\mu$. 

These five algorithms were run on 2000 simulations of uniformly randomly chosen vectors $x$.  A simulation was counted as successful if it recovered the support of $x$ exactly for all $k$ entries.  The results for values of $k$ between 5 and 100 are shown below in Fig.~\ref{fig:success}.  There is a substantial improvement when introducing known structure, even though the coherence theory does not fully apply.  It is clear from this experiment that additional structure helps in recovery. This entire experiment was then repeated with the additional step of adding Gaussian white noise to the data, at a level of $5\%$ with respect to the data.  The results are also shown in Fig.~\ref{fig:success}, where the change in probability of recovery (relative to the original probability $P^{(0)}$ in (A)) is plotted. This was defined as 
\begin{equation}
    \label{eq:rel_success}
    \delta P = \frac{P^{(noise)}-P^{(0)}}{P^{(0)}} \ ,
\end{equation}
where $P^{(noise)}$ denotes the probability of recovery when 5\% noise was added.
One can see the robustness of both IHT and structured IHT with respect to noise, but most importantly, we see that having refined structure ($L=20$) is more robust to noise compared to having less structure information $(L\le5)$.  We note that the implementation of structured IHT used did not employ parallelization, as there was not much benefit for the small number of groups.  Each iteration of structured IHT runs quickly (about 0.005 seconds) on a reasonable workstation.
 \begin{figure}[h!]
\centering
\begin{subfigure}[b]{0.49\textwidth}
                \centering
                \includegraphics[width=\textwidth]{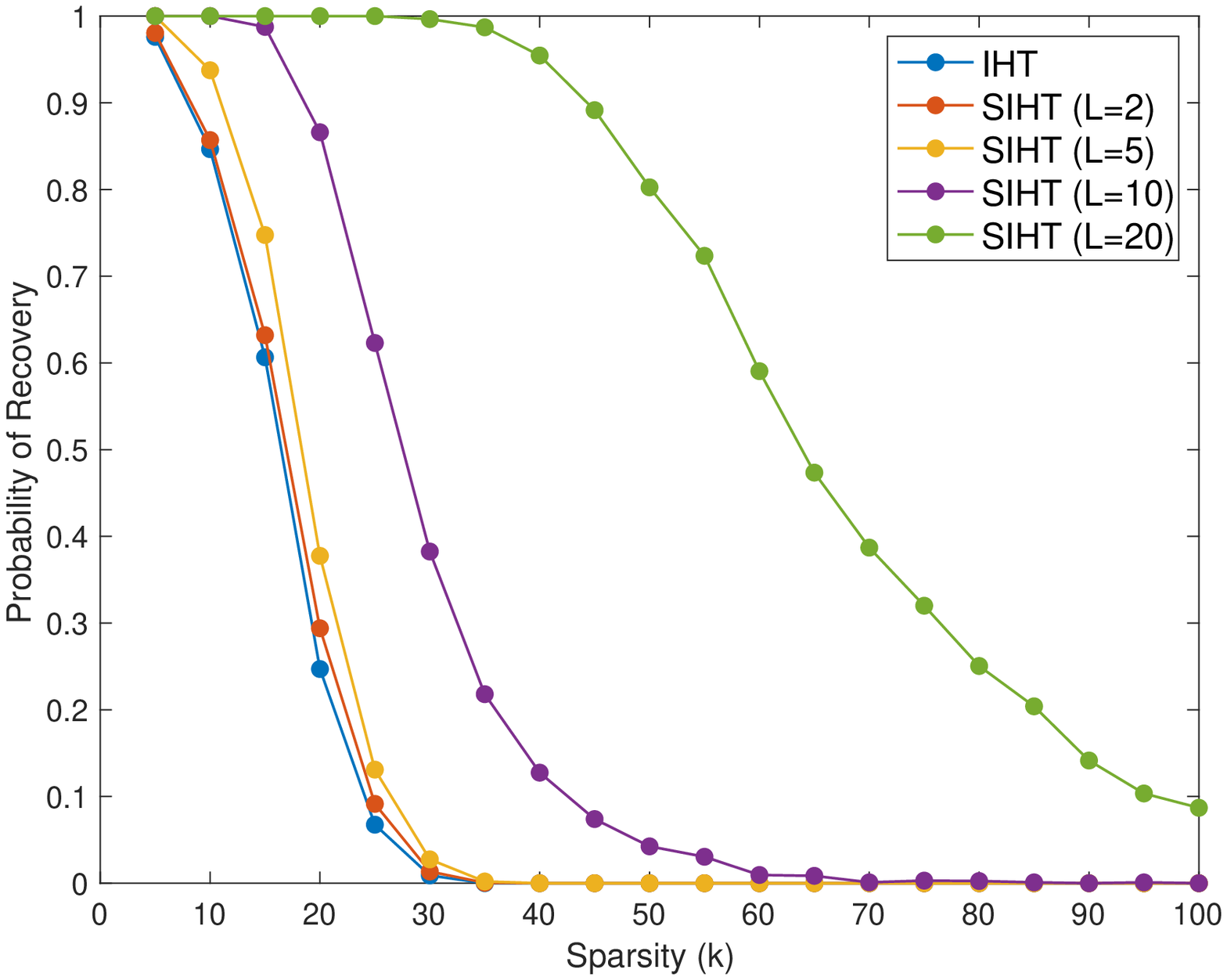}
                \caption{$P^{(0)}$ (no noise)}
        \end{subfigure}
\begin{subfigure}[b]{0.49\textwidth}
                \centering
                \includegraphics[width=\textwidth]{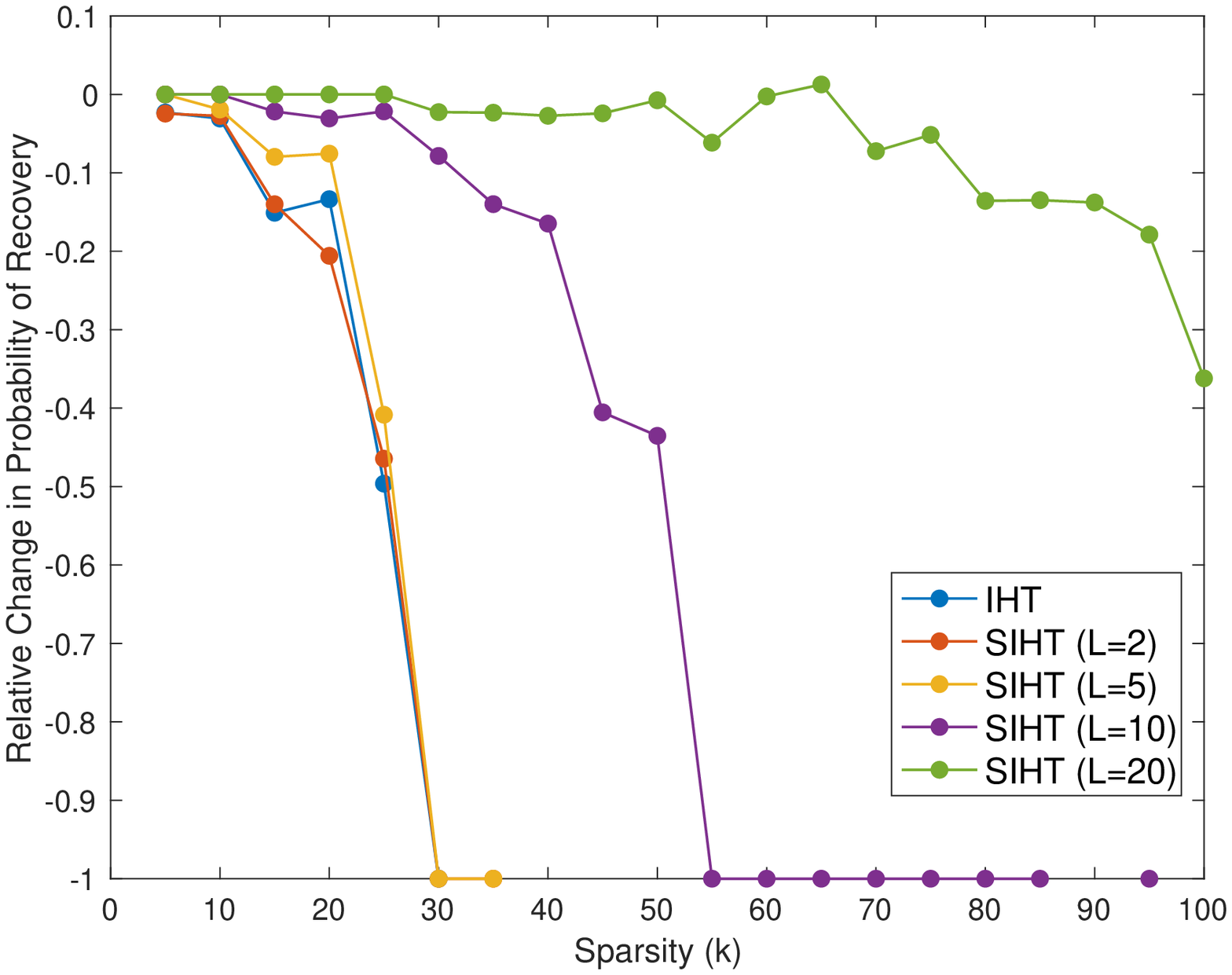}
                \caption{$\delta P$ (5\% noise)}
        \end{subfigure}
\caption{ \label{fig:success} Plots of the probability of successful recoveries for 2000 simulations.  A recovery was counted as successful when the support was recovered exactly.  The plots compare 5 algorithms: IHT and structured IHT (with varying $L$) where the vector is decomposed into $L$ components.  The plot on the left had no noise added to the data.  The plot on the right depicts the relative change in this probability $\delta P$ according to \eqref{eq:rel_success} when Gaussian white noise was added at a level of 5\% relative to the data.  Markers at -1 indicate that the noiseless case had a nonzero probability while the case with noise had a 0 probability.  No markers are shown for cases when both algorithms had a 0\% probability in the simulation.
} 
\end{figure}

Having seen evidence that structured IHT can offer significant improvements over IHT, we now address the second question on how the structure can be chosen in practice.  One key concept is to add a preprocessing step before running the structured IHT algorithm. In the following examples, we will use an initial least-squares recovery (without any sparsity constraints) to help form the underlying structure of the unknown vector.  This combined with underlying knowledge of the problem can help form structures so that structured IHT applies.  We remark that in some cases without any heuristics or additional {\it a priori} information, it may be more difficult to create a reasonable structure for structured IHT.  

We consider the same geometric setup, but remove the simplification of the azimuthal angle $\phi_j$ being fixed at $\pi/2$.  We thus create a 2D grid in $(\theta,\phi)$ to represent the candidate angles $\Phi$.  This is done by taking uniform grids of $N_1$ angles in $\theta$ and $N_2$ angles in $\phi$.  With $N=N_1N_2$, the candidate angles $\Phi_j$ are given by
\begin{align}
\label{eq:2Dgrid_angles}
&\Phi_j=(\theta_m,\phi_n) \ \ ; \ \  j=m+N_1(n-1)  \nonumber \\
&\quad\theta_m =  2\pi m/ N_1 \ \ , \ \  1\le m \le N_1  \nonumber \\
&\quad\phi_n=  \pi n/ N_2 \ \ , \ \  1\le n \le N_2 \ .
\end{align}
We set $N_1=40$ and $N_2=20$, for a total of $N=800$ candidate source angles with $M=100$ detectors placed uniformly about the sphere.  The dimensionless parameter was reduced to $\omega R=10$, which equates to lowering the frequency and decreasing resolution.  No noise was added to the data for this underdetermined problem.  

For the first experiment, $k=5$ sources were given incident angles that were relatively spaced out, with varying amplitudes between 0.4 and 1.  This model is shown in the leftmost plot of Fig.~\ref{fig:ongrid1_model}.  The next plot shows the regularized least squares solution to this underdetermined problem.  While the recovered amplitudes are significantly smaller than the true values in this least squares solution, one can make out 5 separate regions of intensity.  If the sparsity level is known to be $k=5$, it is reasonable to assume that there is one true incident angle in each of these regions.  To this end, we create a thresholded mask that keeps all reconstructed values larger than 7.5 times the mean value in the reconstruction (here, the threshold value was 0.0297).  We remark that this is just one heuristic for choosing a threshold, which could potentially fail for other examples.  However, it worked consistently in our numerical simulations, as all values were kept that were substantially above the mean.  These are the values that most likely contain the support of  a sparse solution.  We see that by using simple post-processing on the least squares solution, we were able to obtain 5 reasonable sets on which to conduct structured IHT.  

\begin{figure}[h!]
\centering
\begin{subfigure}[b]{0.3\textwidth}
                \centering
                \includegraphics[width=\textwidth,trim={4.5cm 0 5cm 0},clip]{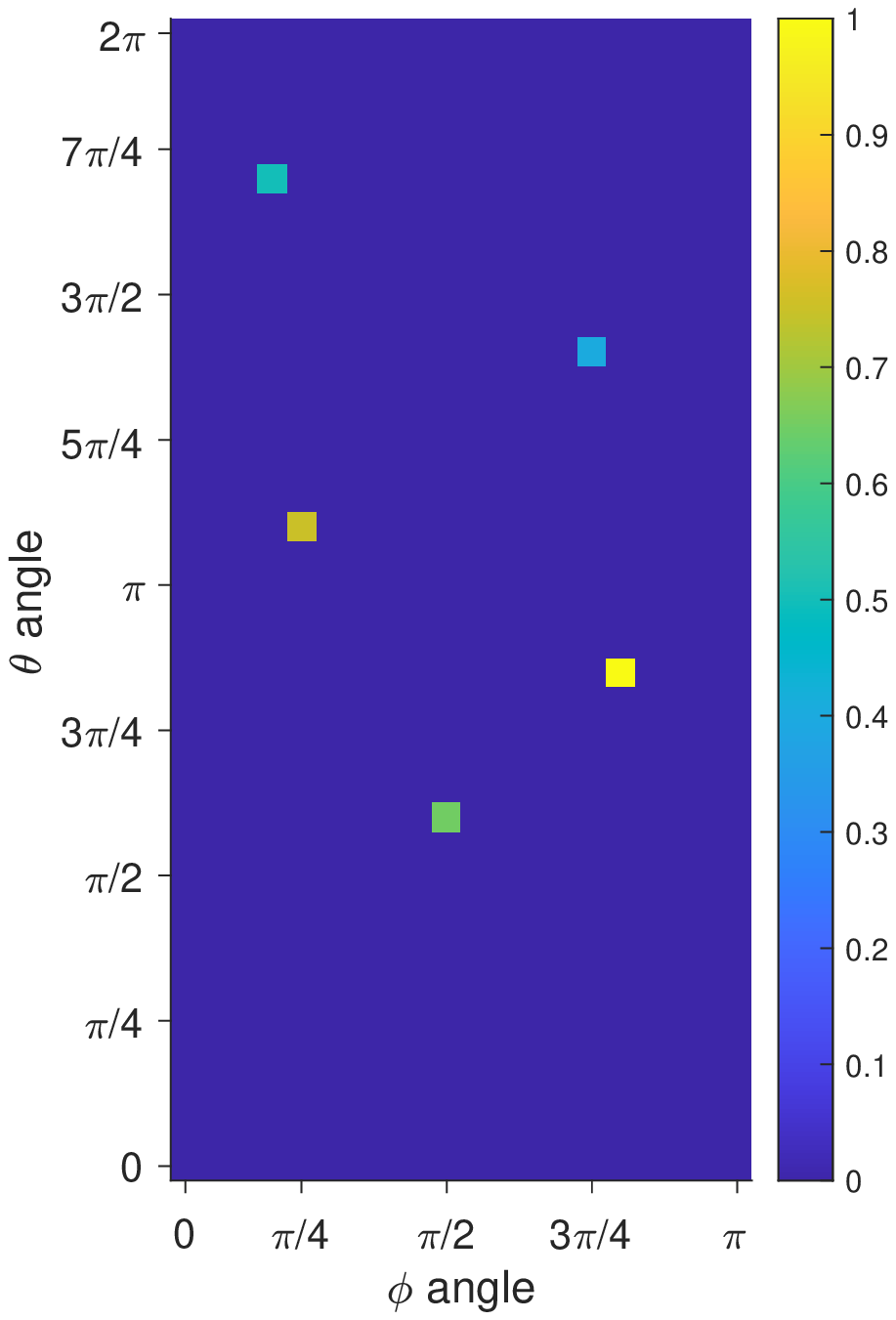}
                \caption{Model 1}
        \end{subfigure}
\begin{subfigure}[b]{0.3\textwidth}
                \centering
                \includegraphics[width=\textwidth,trim={4.5cm 0 5cm 0},clip]{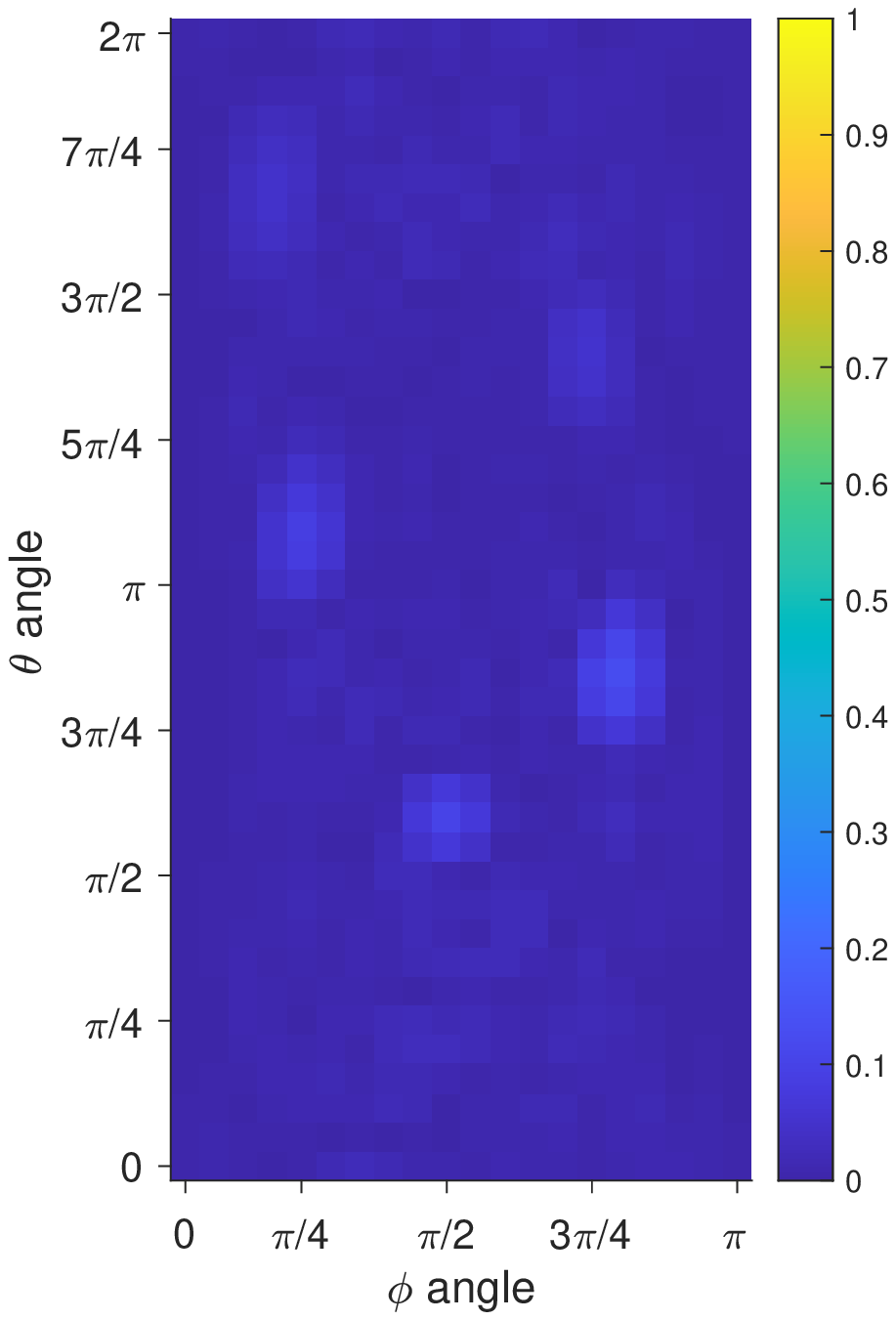}
                \caption{$L^2$ Reconstruction}
        \end{subfigure}
        \begin{subfigure}[b]{0.3\textwidth}
                \centering
                \includegraphics[width=\textwidth,trim={4.5cm 0 5cm 0},clip]{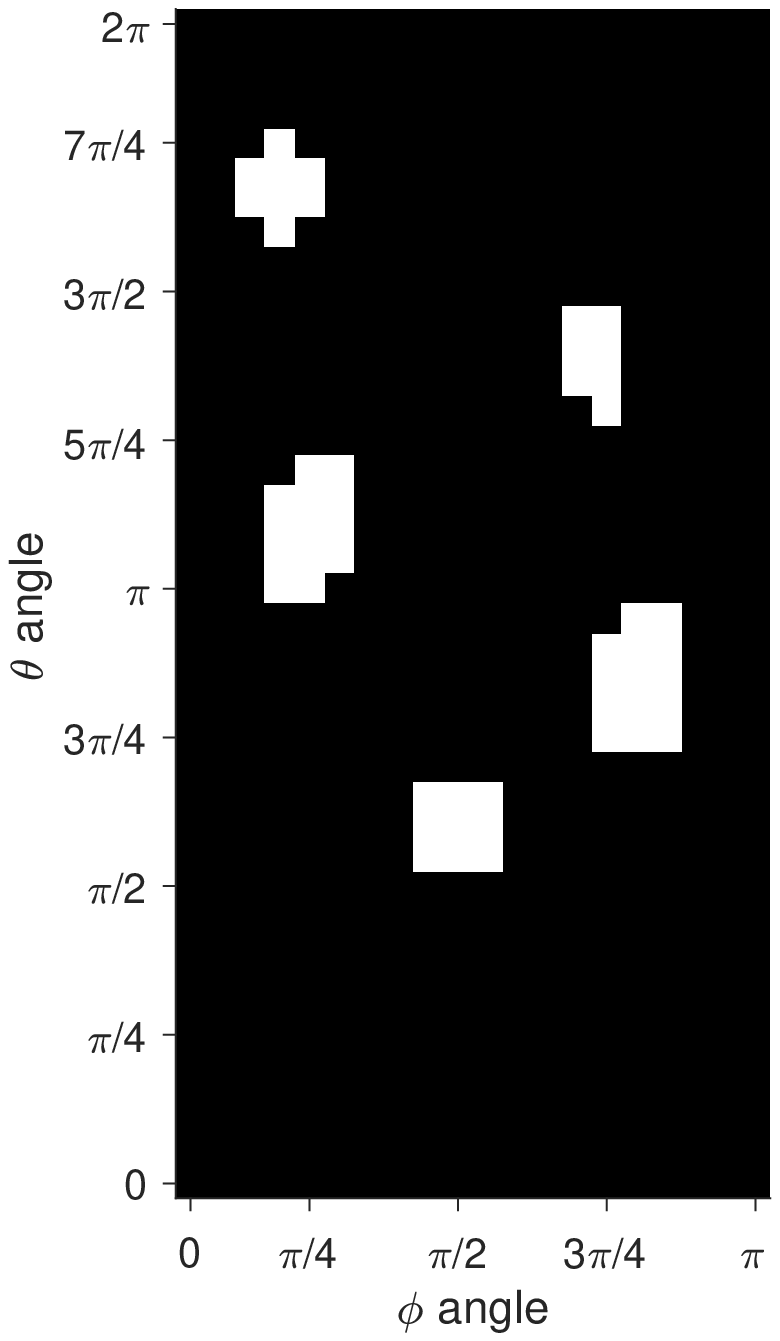}
                \caption{Masked Reconstruction}
        \end{subfigure}
\caption{ \label{fig:ongrid1_model} The model of $k=5$ sources used for the numerical simulation is on the left. The plotted values are the amplitudes for the given incident angle.  The center image is the regularized least squares solution.  On the right is a mask threshold of the values in the least squares solution that are larger that 0.0297.  This yields 5 regions which were then used for structured IHT.}
\end{figure}

We performed two recovery attempts on this model, first with IHT with $k=5$, then with structured IHT on the 5 depicted sets with $k_j=1$ on each set.  In this setup, the coherence of the matrix $A$ was 0.9990, the maximum coherence restricted to a single $S_j$ was 0.9471, and the maximum restricted coherence between any $S_j$ and $S_{j'}$ was 0.2428.  These values are all well outside theoretical bounds, but do still indicate a benefit to using structured IHT. The reconstructions for the two methods are shown in Fig.~\ref{fig:ongrid1_rec}.  The reconstruction for IHT is on the left, which failed to accurately recover the 5 sources.  In contrast, the structured IHT algorithm accurately reconstructed the source information.  We see here that the preprocessing step to generate the sets for structured IHT was powerful.  It is interesting that IHT failed for this fairly simple experiment -- one explanation is that IHT struggles to recover the three weaker (in amplitude) sources that are overshadowed by the two strongest sources (which are incorrectly recovered as multiple sources in close proximity).  Structured IHT avoids this pitfall as it is known that there is only one source in each location.  In fact, this experiment was rerun with all sources having equal amplitude.  This time, both IHT and structured IHT converged.
\begin{figure}[h!]
\centering
\begin{subfigure}[b]{0.35\textwidth}
                \centering
                \includegraphics[width=\textwidth,trim={4.5cm 0 5cm 0},clip]{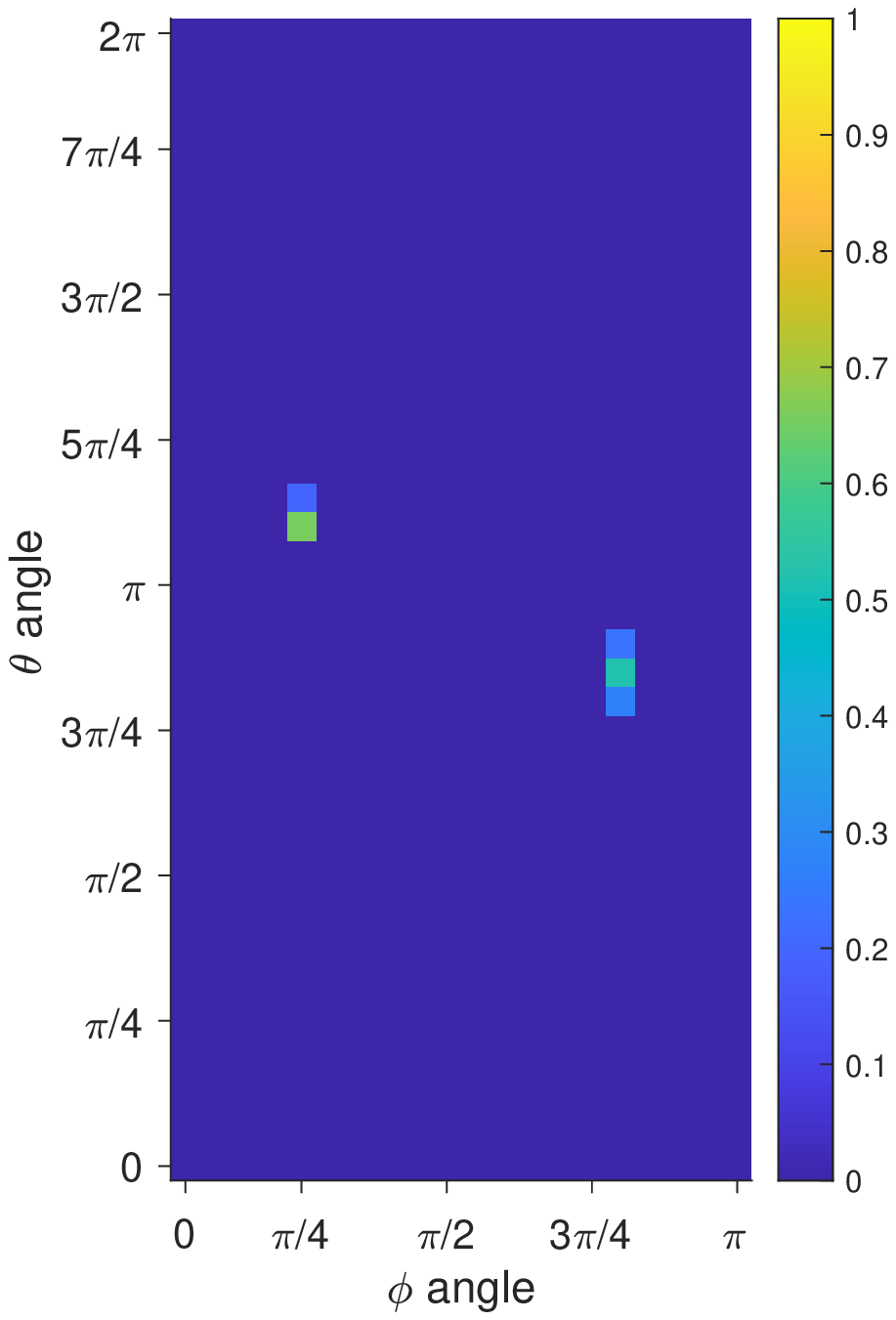}
                \caption{IHT}
        \end{subfigure}
\begin{subfigure}[b]{0.35\textwidth}
                \centering
                \includegraphics[width=\textwidth,trim={4.5cm 0 5cm 0},clip]{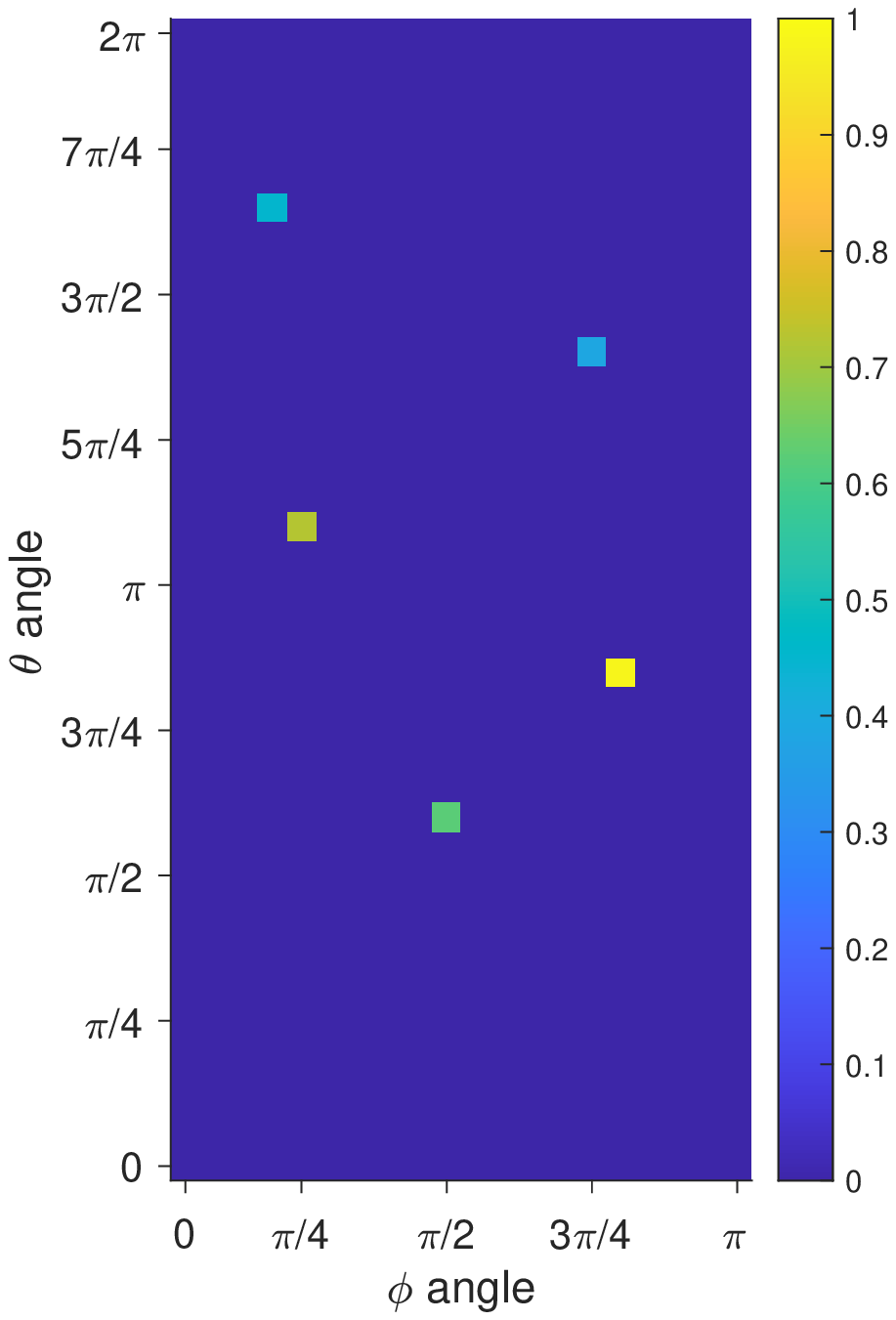}
                \caption{Structured IHT}
        \end{subfigure}
\caption{ \label{fig:ongrid1_rec} The reconstructions for IHT (left) and structured IHT (right) for the model and structured sparsity shown in Fig.~\ref{fig:ongrid1_model}.  The plotted values are the amplitudes for the given candidate angle.  The IHT algorithm failed to converge, whereas the structured IHT algorithm converges.}
\end{figure}

We look at one remaining example to show how preprocessing can allow us to determine effective sets on which to run structured IHT.  Consider the following model of 8 sources depicted on the left in Fig.~\ref{fig:ongrid2_model}.  Note that in this model, all sources were given constant amplitudes of 1.  In this model, some of the sources are close enough together that they cannot be visually separated in the least squares reconstruction.  However, a masking threshold can still be used to help generate the structured sets.  The mask of all values above a threshold of 0.0691 (chosen in the same manner as in the previous experiment) are shown in the right image.  This time, however, we keep the information of the values above this threshold.

If the sparsity level $k=8$ is known, we still need to figure out how to assign the sparsity structure to the 5 regions.  If we assume that our source amplitudes do not vary too much (but not necessarily that they are all equal), it is practical to do this.  Listing these regions from top to bottom, the total sum of the values is 2.72, 1.05, 0.89, 0.92, 1.69.  This correctly leads us to assign $k_1=3$, $k_2=k_3=k_4=1$, and $k_5=2$.  This setup had a similar coherence profile as in the previous example.  With this additional information, structured IHT was able to accurately reconstruct the model.  IHT was again unable to do so.

\begin{figure}[h!]
\centering
\begin{subfigure}[b]{0.35\textwidth}
                \centering
                \includegraphics[width=\textwidth,trim={4.5cm 0 5cm 0},clip]{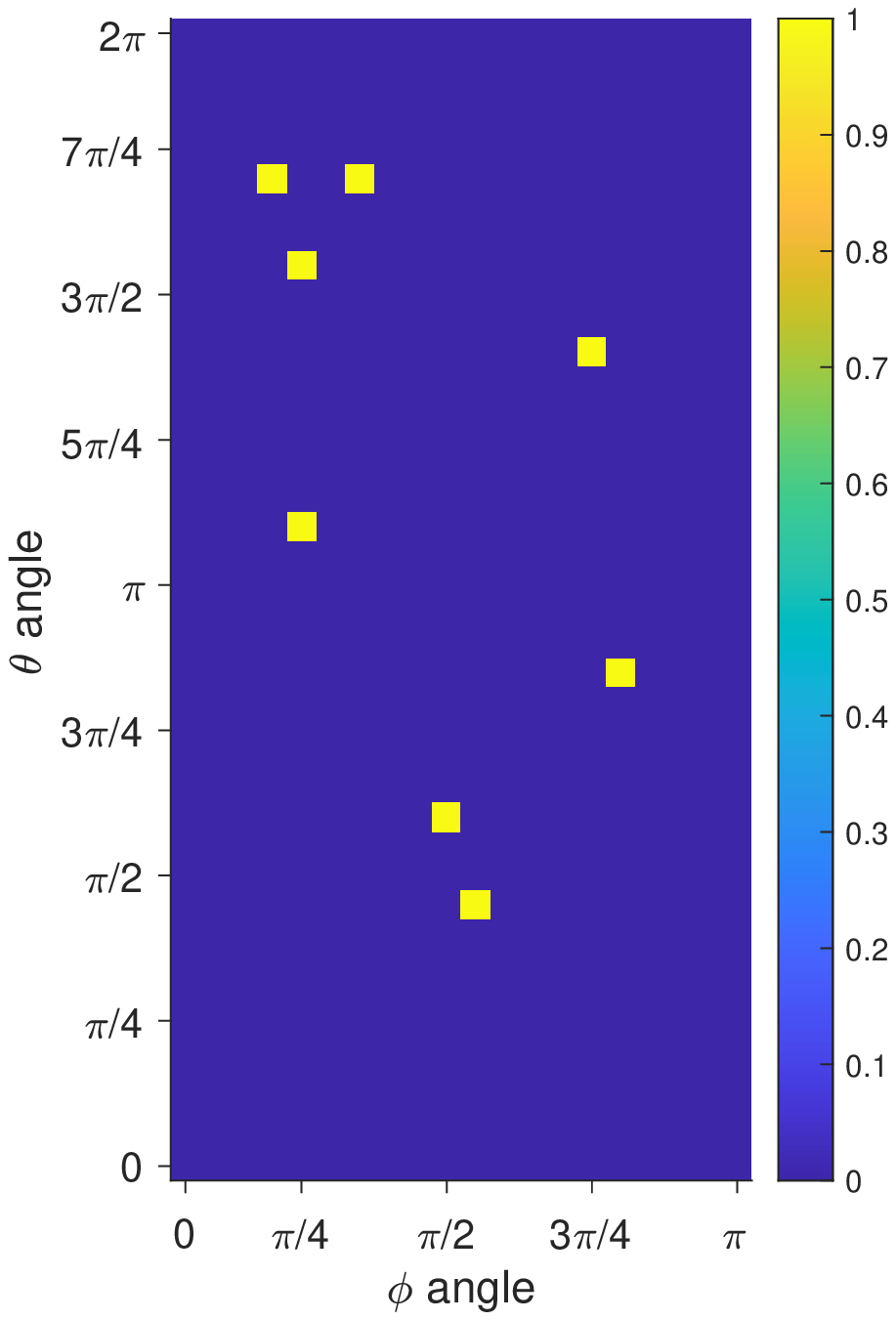}
                \caption{Model 2}
        \end{subfigure}
\begin{subfigure}[b]{0.35\textwidth}
                \centering
                \includegraphics[width=\textwidth,trim={4.5cm 0 5cm 0},clip]{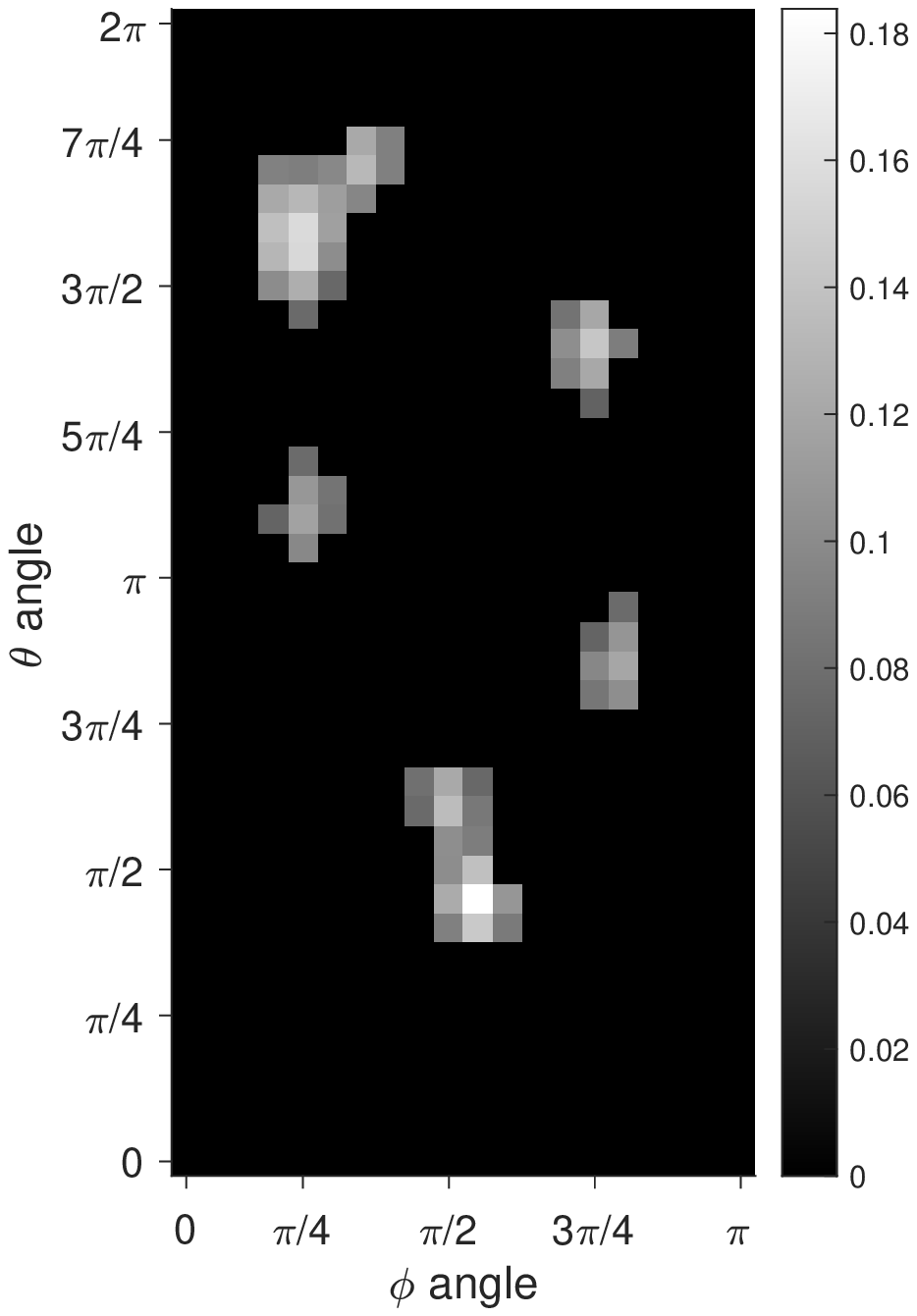}
                \caption{Masked Reconstruction}
        \end{subfigure}
\caption{ \label{fig:ongrid2_model} The model of $k=8$ sources used for the second numerical simulation is on the left. The plotted values are the amplitudes for the given candidate angle and are all fixed at 1 for this model.  The image on the right is the thresholded image of the least squares reconstruction where all values less than 0.0691 were set to 0.  This image was used to generate the 5 regions and their sparsity levels for for structured IHT. }
\end{figure}

Lastly, we remark that the theory of structured IHT (as well as much of sparse recovery theory in general) relies on the exact sparsity value being known ahead of time.  While this is infrequently known in practice, it is not difficult to adjust algorithms to still perform well.  For instance, even if the true sparsity level is unknown, it is often reasonable to obtain an upper bound on the sparsity level \cite{Lopes13,7506083} (or perhaps even by running the algorithm on varying levels of $k$).  Running IHT with thresholding on that upper bound usually yields good results.  For the structured IHT algorithm, one might increase each $k_j$ to some upper bound.  For Model 1 and Model 2, structured IHT still converged when $k_j$ was set to be 2 larger than its true value for each index set.  We note that interestingly, the IHT algorithm did indeed converge for Model 1 if thresholding kept the 7, 8, or 9 largest elements, which is larger than the true sparsity level $k=5$.  In the next section, we will look at an explicit example where the true sparsity level $k$ is unknown.

\section{Off-Grid Applications}
\label{sec:offgrid}
\subsection{Algorithm for Off-Grid Recovery}
If a true source location $\Theta_j$ is not one of the candidate angles $\Phi_{j'}$, one can either accept some level of error in the recovery, or one can attempt to improve the accuracy by recovering the location off-grid.  One natural idea for off-grid recovery is to refine the grid, especially in regions which likely contain the support of the solution.  This must be balanced with not adding too many grid points which could cause numerical instability \cite{electronics8111209}.  

Coherence intuitively plays a role here and also demonstrates how the structured IHT theory can be instructive.
Separate from computational requirements and issues of numerical stability, using a fine grid everywhere in order to minimize the off-grid error significantly increases the coherence.  As seen in \eqref{eq:coh_A}, and common in many other applications, the coherence increases as the grid spacing decreases.  Thus `global' coherence measurements will not yield many theoretical guarantees, which allows for the possibility that the algorithm will not converge.  However, if the true values of the incident angles were well separated, the coherence of the matrix from \eqref{eq:isp_2} (which only includes the correct incident angles and no other test angles) would be significantly smaller.  The idea for the proposed method is thus that one can start with a coarse grid with reasonable coherence.  As the grid is refined, it is ideally refined only around the true source locations.  In this sense, even though $\mu$ and $\mu_{S}$ will increase, the coherence values of the form $\mu_{S,S'}$ should decrease.  If we can construct these sets such that for each set the corresponding sparsity level is 1, according to the theory, it is more likely that the algorithm will converge.  Of course, if the true incident angles are close together, there is less that can be done in terms of coherence.  

We now apply the framework of structured IHT to refine the grid appropriately.  This will be described in a general setting, which will later be applied to the ISP.   Suppose that there are $k$ locations one is attempting to recover in a bounded region $\Omega$ in $\mathbb{R}^d$. Moreover, assume that each of the true locations $\mathbf{x}_j$ are known to live in a subset $S_j\subset \Omega$ such that $S_i\cap S_j=\emptyset$ for all $i$ and $j$.  A coarse grid can then be applied to each of these subsets, in order to run structured IHT.  As it is known that there is exactly one location in each subset, the sparsity level for each subset is  $k_j=1$.  Note that since it is unlikely that the true location is one of the grid points, we do not expect the solution vector to be exactly $k$-sparse, even though it is known that there are $k$ nonzero values.  The hope is that the structured IHT algorithm finds the grid point that is closest to the true location in each subset.  This can then be used as a focal point for refining the grid. 

For instance, let $\mathbf{x}_1,\dots,\mathbf{x}_k\in\mathbb{R}^d$ be the true locations of some unknown quantities such that $\mathbf{x}_j\in S_j.$ Suppose for each $1\le j\le k$, we have a grid of $N_j$ candidate locations $\{\mathbf{y}^{(j)}_i : 1\le i \le N_j\}$, where for all $i$, $\mathbf{y}^{(j)}_i\in S_j$.  We will assume that for all $j$, these grid points have a uniform spacing of $h_1,\dots,h_d$ in the respective coordinate direction.  Suppose after running structured IHT on each $S_j$ with $k_j=1$, we are left with the $k$ grid points $\{\mathbf{y}_{i_1}^{(1)},\dots,\mathbf{y}_{i_k}^{(k)} \}$, where each $i_j$ is some number between 1 and $N_j$.  We now propose running structured IHT again, but on a finer grid (by a factor of $\alpha$, where $1<\alpha\le2$) about these nonzero location.  On the next run of structured IHT, each set $S_j$ uses the grid points
\begin{equation}
\label{eq:refine}
   S_j= \left\{ \mathbf{y}^{(j)}_{i_j} \pm (a_1,\dots,a_d) : a_i \in\{0,\pm h_i/\alpha\}   \right\} \ .
\end{equation}
This process can be repeated until convergence or some stopping criteria.  This algorithm is described below as Algorithm \ref{algo:offgrid1}.  This description can easily be modified to allow for overlapping grids or keeping more than just one grid point with $k_{j}>1$.  For simplicity, we will stick with the described method. \\

\begin{algorithm}[H]
\label{algo:offgrid1}
\SetAlgoLined
 {\bf Input}: Data $b$, sparsity level $k$ \\
  {\bf Input}: Disjoint sets $S_1,\dots,S_k\subset\mathbb{R}^d$ \\ 
  {\bf Input}: Initial grid $\mathbf{Y}^{(1)}=\{\mathbf{y}^{(j)}_i : 1\le i \le N_j, 1\le j\le k\} $ \\
    {\bf Input}: Initial grid spacings $h_1^{(1)},\dots,h_d^{(1)}$ and refinement parameter $\alpha$  \\
 {\bf Output}: $\mathbf{Y}$ containing $k$ locations in $\mathbb{R}^d$  \\
  {\bf Output}: $x$ containing $k$ associated amplitudes  \\
 \vspace{3mm}
Set $t=1$\\
 \While{stopping criteria not met}{
  Generate $A(\mathbf{Y}^{(t)})$ using grid $\mathbf{Y}^{(t)}$\\
  \vspace{0.5pc}
  Run Structured IHT for $A(\mathbf{Y}^{(t)})x^{(t)}=b$ on $S_j=\{\mathbf{y}_i^{(j)} \}$ with $k_j=1$ for all $j$\\
 \hspace{2pc} This outputs $k$ locations $\{\mathbf{y}_{i_1}^{(1)},\dots,\mathbf{y}_{i_k}^{(k)} \}$\\
  \vspace{0.5pc}
 Create refined grid: $\mathbf{Y}^{(t+1)}=\left\{ \mathbf{y}^{(j)}_{i_j} \pm (a_1,\dots,a_d) : a_i \in\{0,\pm h^{(t)}_i/\alpha,1\le j\le k\}\right\} $\\
	\vspace{0.5pc}
	$h_i^{(t)}\gets h_i^{(t)}/\alpha$ for all $i$\\
	\vspace{0.5pc}
	$t \gets t+1$
 }
 \caption{Structured IHT with Grid Refinement for Off-Grid Recovery}
\end{algorithm} 

We now provide some theoretical justification for this idea.  For this refinement of the grid to improve the accuracy of recovery, it needs to be shown that in each iteration of Algorithm \ref{algo:offgrid1}, the structured IHT algorithm will converge to the available grid location that is closest (in physical distance) to each, potentially off-grid, true source location.  We emphasize that in this grid refinement process, it is more important to find the correct support as opposed to the correct nonzero values.  

Let $\mathbf{X}=\{\mathbf{x}_1,\dots\mathbf{x}_k\}$ contain the support locations in $\mathbf{R}^d$ of $x^*$.  That is, $x^*$ is a vector of length $k$ with all nonzero entries (which we will call the amplitudes), but each entry of $x^*$ has a real-space associated location given by $\mathbf{X}$.  Consider the linear system of the form $A(\mathbf{X})x^* =b$, where $A(\mathbf{X})$ is a matrix whose columns depend on the locations $\mathbf{X}$ (as in \eqref{eq:isp_3} for example).  The vector $b$ is the available data.  We assume we already have non-overlapping sets $S_j$ such that $\mathbf{x}_j\in S_j$.  Suppose we have a grid of candidate locations $\mathbf{Y}=\{\mathbf{y}^{(j)}_i : 1\le i \le n_j, 1\le j \le k\}$, where for all $i$, $\mathbf{y}^{(j)}\in S_j$.   Define the set $\mathbf{\tilde{X}}$ to contain the grid angles closest (in physical distance) to the true, potentially off-grid, source locations.  This set is defined as
\begin{equation}
\mathbf{\tilde{X}} = \left\{ y^{(1)}_{m_1}, \dots, y^{(k)}_{m_k}  \right\} \text{  where  } m_j=\argmin_{j'} \|\mathbf{y}^{(j)}_{j'}- \mathbf{x}_j\|_2  \text{ for all } j\ .
\end{equation}
Our goal is to ensure that structured IHT will  converge to a vector that has support equal to $\mathbf{\tilde{X}}$.  Let $\tilde{x}^*$ be the minimum $\ell_0$ solution to the system  
\begin{equation}
\label{eq:min1}
   \tilde{x}^* = \min_{x} \|b - A(\mathbf{\tilde{X}})x\|_0 \ .
\end{equation}
The relevant sparse linear system to solve is now 
\begin{equation}
    b = A(\mathbf{Y})\tilde{x}^*+\epsilon \ ,
\end{equation}
where the error term is given simply by $\epsilon = b-A(\mathbf{Y})\tilde{x}^*$.  By \eqref{eq:min1}, this error has been chosen to be as small as possible in terms of the $\ell_0$ norm (for the fixed support).   We seek to ensure that after running the structured IHT algorithm and obtaining an $x^{(t)}$, we have $\supp(x^{(t)})=\supp(\tilde{x}^*)$.  Applying Corollary \ref{cor1} to this setup, we have
\begin{align}
\|&x^{(t)}-\tilde{x}^*\|_1 \le \left(\rho+(k-1)\tilde{\rho}\right)^t\|x^{(0)}-\tilde{x}^*\|_1  + \frac{3k\|A^*(\mathbf{Y})\epsilon\|_\infty}{1-\rho-(k-1)\tilde{\rho}} \ ,
\end{align}
where we have set $L=k$.  Note that $\rho$ and $\tilde{\rho}$ depend on the matrix $A(\mathbf{Y})$, hence, they also depend on the choice of grid points $\mathbf{Y}$. Assuming that $\rho+(k-1)\tilde{\rho}<1$ (which is the necessary condition for theoretical convergence), the only term that cannot be made arbitrarily small by taking $t$ large enough is the error term.  This fixed term must be small enough to guarantee that $\supp(x^{(t)})=\supp(\tilde{x}^*)$ as $t\to\infty$.  A sufficient, but not necessary, way to guarantee that the supports are equal is if the error $\|x^{(t)}-\tilde{x}^*\|_1$ is smaller than the minimum (in absolute value) entry of $\tilde{x}^*$.  We define this value by
\begin{equation}
    \tilde{c} = \min_j\left\{|(\tilde{x}_*)_j| : |(\tilde{x}_*)_j|> 0 \right\} \ .
\end{equation}
Thus, if $3k\|A^*(\mathbf{Y})\epsilon\|_\infty/(1-\rho-(k-1)\tilde{\rho})<\tilde{c}$, we can run the structured IHT algorithm for sufficiently many iterations to guarantee that the support of the result matches the $k$ candidate locations closest to the true locations.  This is summarized in the following proposition.
\begin{proposition}
\label{prop:off}
Let $\{\mathbf{Y}^{(t)}\}$ be the sequence generated from the structured IHT algorithm with grid refinement (Algorithm \ref{algo:offgrid1}) for the equation $A(\mathbf{X})x^*=b$, where $\mathbf{X}=\{\mathbf{x}_1,\dots,\mathbf{x}_k\}$, and each $\mathbf{x}_j\in S_j$.
For each $t$, let $S^{(t)}_j$ be the locations in $\mathbf{Y}^{(t)}$ that are contained in $S_j$.  
For each $\mathbf{Y}^{(t)}$, define 
\begin{equation}
\rho^{(t)} = \max_{1\le n\le k} \mu(A(\mathbf{Y}^{(t)}))_{S^{(t)}_n} \quad,\quad \tilde{\rho}^{(t)}=\max_{\substack{1\le m,n\le k\\m\neq n}} \mu(A(\mathbf{Y}^{(t)}))_{S^{(t)}_n,S^{(t)}_m} \ . 
\end{equation}
Suppose there exists a $T>0$ such that for all $t<T$, $\rho^{(t)}+(k-1)\tilde{\rho}^{(t)}<1$, and 
\begin{equation}
\label{eq:grid23}
3k\|A^*(\mathbf{Y}^{(t)})\epsilon^{(t)}\|_\infty/(1-\rho^{(t)}-(k-1)\tilde{\rho}^{(t)})<\tilde{c}^{(t)} \ .\end{equation}
Then for all $t<T$
\begin{equation}
\label{eq:grid12}
    \|\mathbf{Y}^{(t)}-\mathbf{X} \|_2 \le \|\mathbf{Y}^{(t-1)}-\mathbf{X}\| \ .
\end{equation}
\end{proposition}
Note that we can have equality in \eqref{eq:grid12} in cases when $\mathbf{Y}^{(t)}$ is already closer to $\mathbf{X}$ than any of the refined grid points.  In general, we expect $\tilde{\rho}^{(t)}$ to be non-increasing as $t$ increases, whereas $\rho^{(t)}$ is most likely going to increase above 1 for large enough values of $t$.  Additionally, one must start with an accurate enough grid such that the true solution is well represented in order for \eqref{eq:grid23} to be satisfied.

We now apply this algorithm to our ISP for off-grid recovery.  We will also tackle some of the challenging questions in terms of running a sparse recovery algorithm in practice, such as how to perform recovery when the true sparsity level $k$ is unknown.
We consider the same ISP as in the previous section of trying to recover $k$ plane wave sources from measurements of the total field.  In the case when $k$ is known, to use the grid refinement outlined in Algorithm \ref{algo:offgrid1}, we must input $k$ disjoint sets, each of which contains one source location.  However, this is not a simple task, and preprocessing methods such as those described in Section \ref{sec:ongrid} can fail due to the amount of noise introduced from using a coarse grid.  It is important not to be too restrictive with these initial sets, as they might not include the true locations.  To this end, we describe a more involved first step to find the $k$ sets.   

Suppose for now that the true sparsity level $k$ is known, but a partition into $k$ sets which each contain a true location is not known.  Using an initial coarse grid of $N_0$ points, we propose creating $N_0$ initial sets, where each set $S_j$ is centered on one of the grid points $\mathbf{y}_j^{(0)}$.  For each $j$, we create a refined grid via \eqref{eq:refine} and include these $3^d$ points in $S_j$.  Then Algorithm 2 can be run.  The only concern is that this initial grid has to be finer than the minimum spacing between any two sources. When $N_0\gg k$, many of these sets can be in close proximity to one another, allowing for the possibility of the same angle appearing in more than one set $S_j$.  In this scenario one can arbitrarily remove the repeated angle from one of the sets so that all sets are mutually disjoint.  

However, if Algorithm \ref{algo:offgrid1} were continuously run on this setup, this would result in finding $N_0$ locations, for which we know only $k$ should exist.  Thus, between each run through of this process, we interject one additional threshold step to reduce the total number of grid points we are keeping from the initial number of sets $N_0$.  In between each run of structured IHT and refining the grid, we threshold  the resulting vector $x^{(t)}$ from structured IHT   by $H_{K_t}(x)$ for some $K^{(t)}\ge k$.  These locations that are thresholded are permanently removed from the grid, along with their containing set $S_j$.  The remaining question is how to choose $K^{(t)}$.   Similar to our masking procedure in the previous section, we propose a threshold that is based on the mean of remaining values and threshold those that are insignificant.  In this scenario, for some positive constant $c^{(t)}<1$
\begin{equation}
\label{eq:Kt_def}
K^{(t)}=\max\left\{\left| \left\{x^{(t)}_\ell : x^{(t)}_\ell > c^{(t)} M_1^{(t)} \right\}\right| ,k\right\}\qquad,\qquad M_1^{(t)} = \frac{1}{|x^{(t)}|}\sum_{j=1}^{|x^{(t)}|} \left|x^{(t)}_j\right| \ .
\end{equation}
If $k$ is known, the definition of $K_t$ has a floor of $k$. Once $K_t=k$, the algorithm continues exactly as stated in Algorithm \ref{algo:offgrid1}.  The average value of the vector $x^{(t)}$, denoted by $M_1^{(t)}$, changes each iteration, and one has much freedom in choosing $c^{(t)}$.  In our implementation, for simplicity, we have left $c^{(t)}$ unchanged each iteration.  If $c^{(t)}$ is too large, the value of $K^{(t)}$ can be too small and lead to failed convergence in cases of  high variance between source.    However, these problems are more difficult in general as the weaker source can be easily lost in any noise that is present.

What we have proposed is a general framework for off-grid recovery.  It alternates between running structured IHT to refine the grid in an optimal direction, and a global threshold to narrow down the size of the support to the true sparsity level.  Note that there is substantial flexibility in the algorithm design regarding how to refine the grid, and how to determine the sparsity levels and grouping schemes for the subsequent structured IHT runs. We have described a sample implementation, which is summarized below as Algorithm \ref{algo:offgrid}.\\

\begin{algorithm}[H]
\label{algo:offgrid}
\SetAlgoLined
 {\bf Input}: Data $b$\\
 {\bf Input}: Overall sparsity level $k$ {\it(optional)}\\
 {\bf Input}: Disjoint sets $S_1,\dots,S_{N_0}\subset\mathbb{R}^d$\\
   {\bf Input}: Initial grid $\mathbf{Y}^{(1)}=\{\mathbf{y}^{(j)}_i : 1\le i \le N_j, 1\le j\le k\} $ \\
    {\bf Input}: Initial grid spacings $h_1^{(1)},\dots,h_d^{(1)}$  \\
    {\bf Input}: Parameter $0<c<1$ \\ 
 {\bf Output}: $x$\\
 \vspace{3mm}
Set $t=1, N=N_0$\\
 \While{stopping criteria not met}{
  $(\mathbf{Y}^{(t+1)},x^{(t)})$=Algorithm 2($b,S_j,\mathbf{Y}^{(t)},h_j^{(t)})$ run for one iteration \\
  \vspace{0.5pc}
  $M_1^{(t)}=\frac{1}{N}\sum_{j=1}^{N} \left|x^{(t)}_j\right|$ \\
  \vspace{0.5pc}
	$K^{(t)}=\max\left\{\left| \left\{x^{(t)}_\ell : x^{(t)}_\ell > c M_1^{(t)} \right\}\right|,k \text{  (if known)}\right\}$\\
	\vspace{0.5pc}
	$x^{(t)}=H_{K^{(t)}}(x^{(t)})$\\
	\vspace{0.5pc}
	\For{$j = 1$ to $N$}
	{\If{$S_j\cap$ real-space support of $x^{(t)}$=$\emptyset$ }{\vspace{0.3pc}
	Remove $S_j$ from list and associated grid points from $\mathbf{Y}^{(t+1)}$\\
	\vspace{0.5pc}
	$N\gets N-1$ 
	}}
	\vspace{0.5pc}
	$t \gets t+1$
 }
 \caption{Structured IHT for Off-Grid Recovery}
\end{algorithm} 
\vspace{1pc}

The convergence of Algorithm 3 relies on this process eventually converging to Algorithm 2, when one can then apply the results of Proposition \ref{prop:off}. To that end, one must ensure that this threshold never removes a correct set, and that it will eventually go down to $k$ sets.  One way to ensure that this occurs depends on the accuracy of $\tilde{x}^*$ from \eqref{eq:min1} in terms of approximating the true amplitudes.  For instance, let $z^*$ be the vector that has the correct amplitudes of $\tilde{x}^*$ in the matching entries of $\mathbf{\tilde{X}}$ (with zeros as appropriately needed).  Then we have
\begin{align}
    \|x^{(t)}-z^*\|_1 =  & \|x^{(t)}-\tilde{x}^*+\tilde{x}^*-z^*\|_1 \nonumber \\ 
    \le &  \|x^{(t)}-\tilde{x}^*\|_1 + \|\tilde{x}^*-z^*\|_1 = E  \ .
    \end{align}
The first of these error terms is controlled by the result of Corollary \ref{cor1}, while the second term depends on the solution to \eqref{eq:min1}.  As the farthest any entry of $x^{(t)}$ can be from $z_*$ is E, if $c^{(t)}M_1^{(t)}$ is less than the minimum value of $\tilde{x}^*$ minus $E$ , then we are guaranteed not to eliminate any entries that should stay. On the other hand, if $E<c^{(t)}M_1^{(t)}$, then any entries that should be 0 will be removed.  If both of these inequalities can be satisfied, then there exists a $c^{(t)}$ such that Algorithm 2 can be applied.

While the case when the true sparsity level $k$ is unknown can be more difficult, we propose using this described algorithm with only minor modifications.  The only instance when the sparsity level $k$ is required in Algorithm 3 is for setting a floor value for $K^{(t)}$ in \eqref{eq:Kt_def}.  When $k$ is unknown, one can either set the floor for $K^{(t)}$ to be 1, or to some known lower bound that is problem-specific.   Using Algorithm \ref{algo:offgrid} for off-grid recovery when the true sparsity level $k$ is unknown will be looked at in the following numerical simulations.

\subsection{Numerical Simulations}
We first test Algorithm 2 to demonstrate the use of the off-grid recovery algorithm.  We begin with a one-dimensional experiment in polar angle $\theta$, where the azimuthal angle is fixed at $\phi=\pi/2$.  Five values of $\theta$ were chosen, where each $\theta_j$ was chosen uniformly at random in the interval $[2(j-1)\pi/5,(2j-1)\pi/5]$.  Each source was given unit amplitude.  The dimensionless parameter was set to be $\omega R=40$ with $M=100$ detectors placed around a sphere.  No noise was added to the data (besides the noise from having an off-grid source).  We conducted three numerical experiments attempting to recover these sources.  

The first algorithm attempted was the standard IHT algorithm on $N$ test sources in the form of \eqref{eq:1d_angles}.  Note that with a probability of 1, none of these test source angles were equal to the true angles.  The second algorithm was the structured IHT algorithm with $N$ test angles, with the given structure $S_j = \{(j-1)N+1,\dots,jN\}$ for $j=1,\dots,5$.  This ensured that $k_j=1$ for each $S_j$.  The third algorithm tested was Algorithm 2 for off-grid recovery using these same sets $S_j$. For this off-grid recovery algorithm, we set $\alpha = 1.1$ in \eqref{eq:refine} and used an initial coarse grid of $N_0=15$ incident angles.  

The results of these experiments are shown below in Fig.~\ref{fig:3graphs}.  For IHT (A) and structured IHT (B), the experiment was run 100 times, for varying values of $N$ between 0 and 1000.  One can see that, in general, as $N$ increases, the coherence increases (depicted by the blue lines).  The red orange line plots the relative residual, given by 
\begin{equation}
    \label{eq:residual}
    \delta r = \frac{\|b-A(\mathbf{Y})x^{(t)}\|_2}{\|b\|_2} \ .
\end{equation}
For IHT, one can see that as the coherence increases beyond a certain point (roughly 0.5), the algorithm becomes unstable.  For some values of $N$, IHT finds a reasonable solution, but by slightly changing this value of $N$, the algorithm fails to converge correctly.  On the other hand, with the additional structure information, structured IHT  has a decreasing trend in the residual as $N$ increases, despite the fact that the coherence is converging to 1.

The last plot in Fig.~\ref{fig:3graphs} shows Algorithm 2 starting with $N_0=15$ grid points.  Plotted are the coherence values and relative residual for 100 iterations of Algorithm 2.  We note that the majority of these coherence values were too large to apply the theoretical results of Proposition \ref{prop:off}.  The shape of these curves resembles the results of the structured IHT algorithm.  However, there is a clear benefit to using Algorithm 2 for this simple example.  The convergence plot for Algorithm 2 is much smoother than the convergence of structured IHT. If one were to run the on-grid structured IHT algorithm, the result would highly depend on the choice of $N$.  For example, choosing $N=91$ results in $\delta r=0.0479$, while $N=96$ doubles the relative residual to 0.0900.  The point is, when running Algorithm 2, the user does not need to choose a fixed $N$, and can instead allow the algorithm to, in a sense, choose an appropriate value for $N$.  

\begin{figure}[h!]
\centering
\begin{subfigure}[b]{0.32\textwidth}
                \centering
                \includegraphics[width=\textwidth]{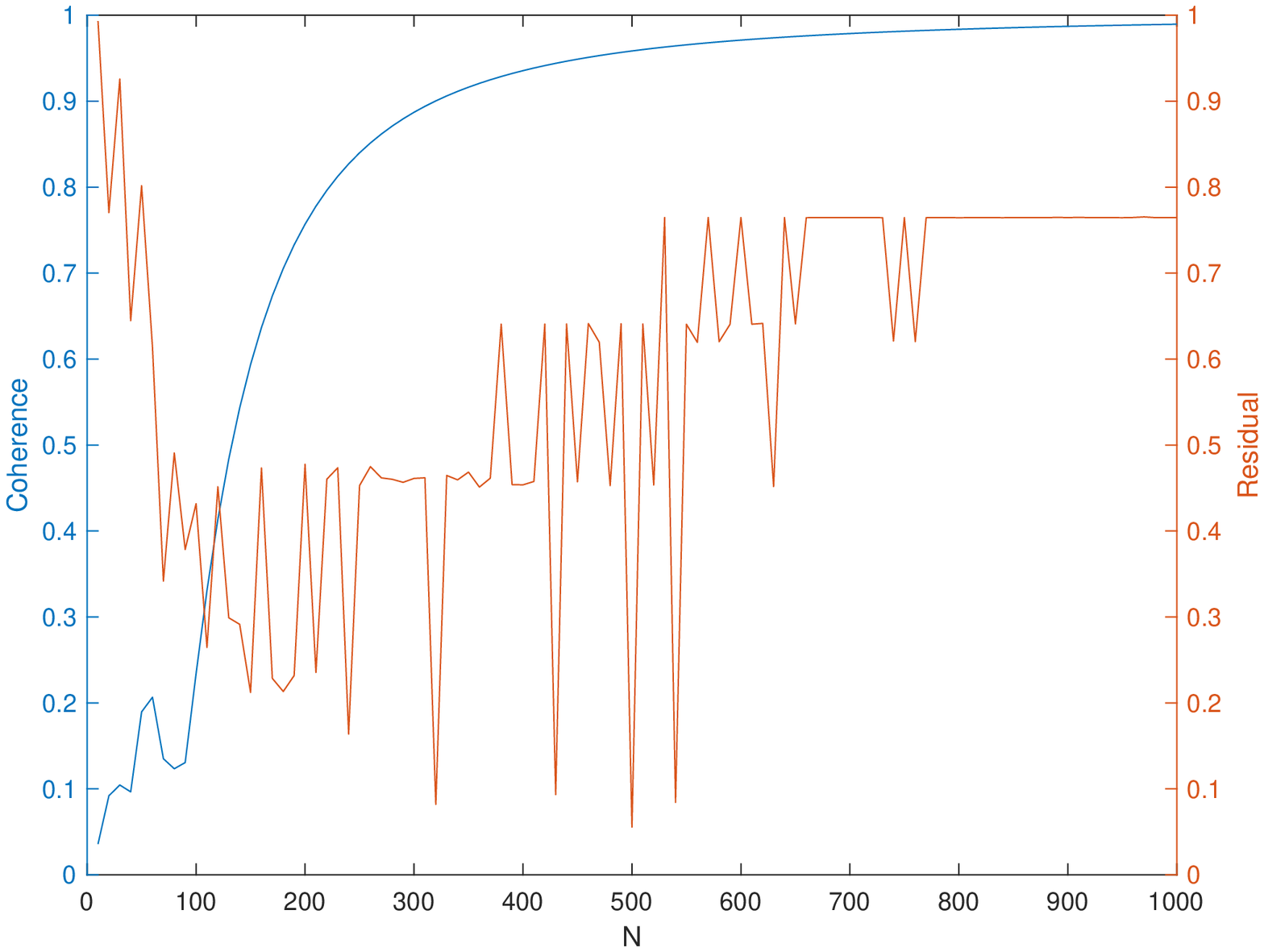}
                \caption{IHT}
        \end{subfigure}
\begin{subfigure}[b]{0.32\textwidth}
                \centering
                \includegraphics[width=\textwidth]{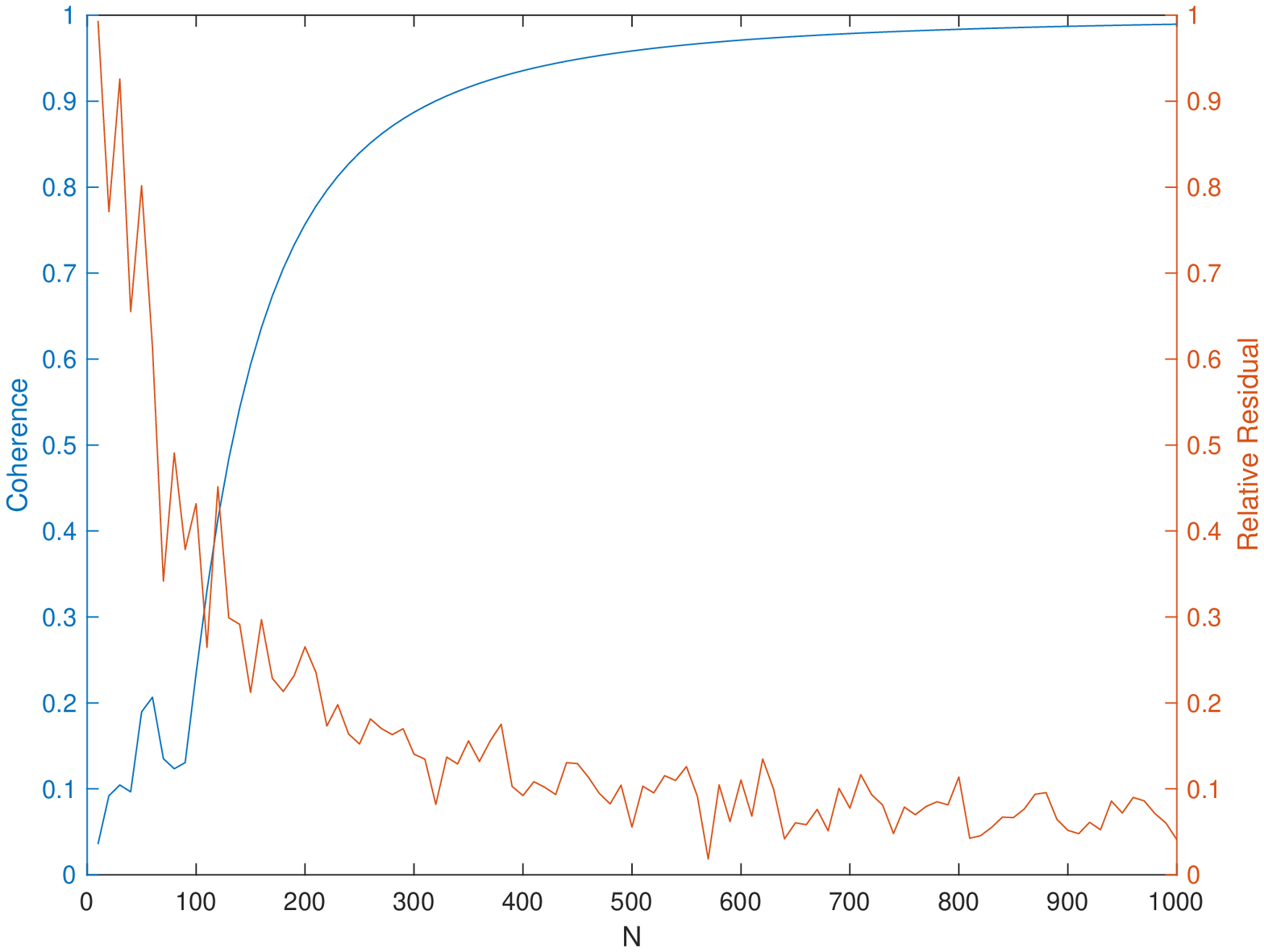}
                \caption{Structured IHT}
        \end{subfigure}
        \begin{subfigure}[b]{0.32\textwidth}
                \centering
                \includegraphics[width=\textwidth]{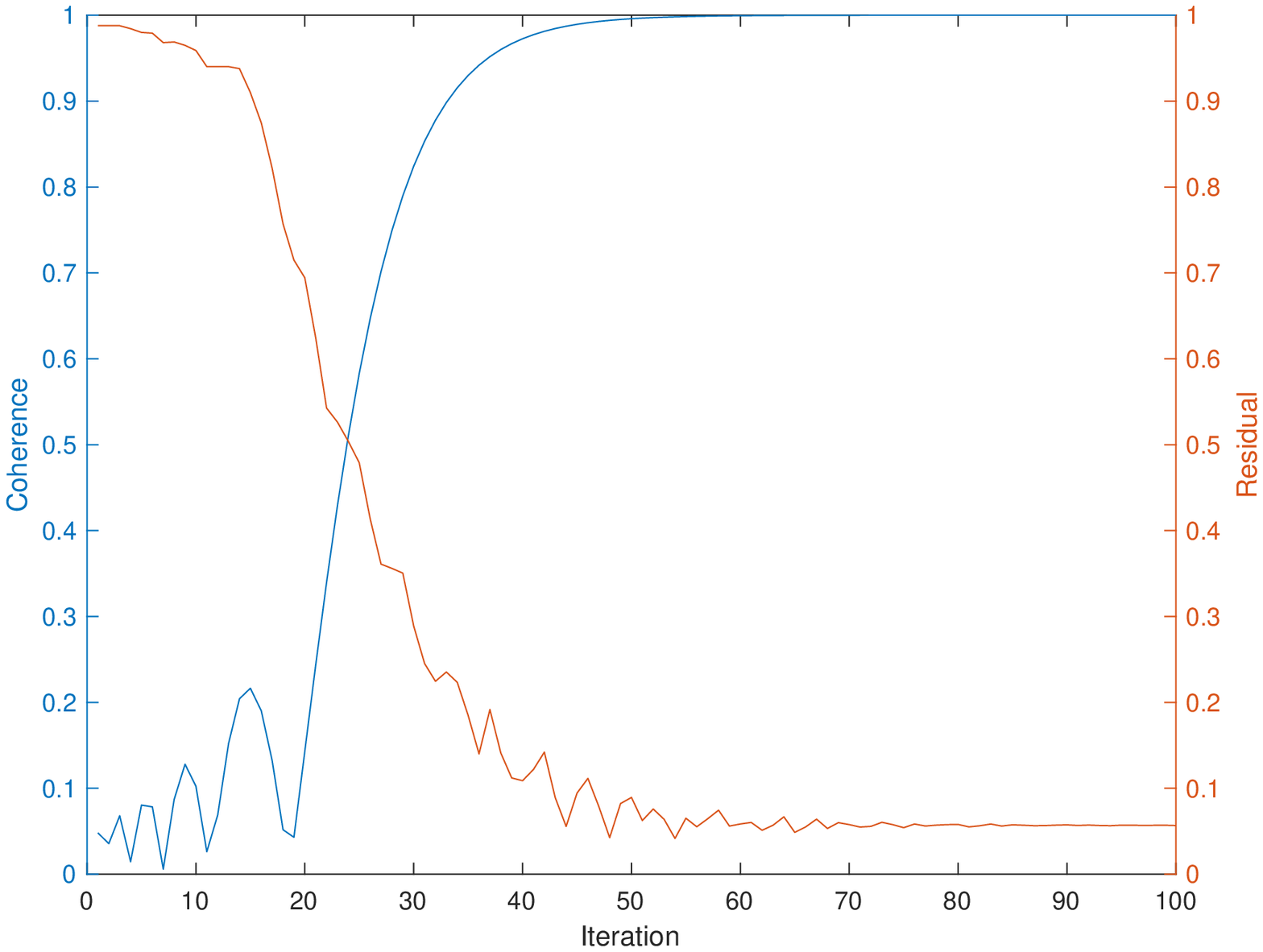}
                \caption{Off-grid recovery}
        \end{subfigure}
\caption{ \label{fig:3graphs} Results for IHT (left) and structured IHT (center) for varying values of $N$ test source angles.  The coherence and relative residual \eqref{eq:residual} are plotted in blue and red orange respectively. The same quantities for 100 iterations of Algorithm 2 are plotted on the right.}
\end{figure}

We then tested the performance of Algorithm 3 for off-grid recovery on a more complicated example.  In particular, we considered the case in which the true sparsity value was unknown.  We consider the full two-dimensional problem in $\theta$ and $\phi$, with $M=100$ detectors placed uniformly around a sphere.  The dimensionless parameter was decreased to $\omega R=4$.  
Six plane waves were placed with angles chosen uniformly randomly and amplitudes chosen randomly from a normal distribution with mean 1 and standard deviation 0.2.  An initial coarse uniform 20 $\times$ 10 grid in $(\theta,\phi)$ is used.  The true angle values and the least square solution interpolated over this grid is shown below in Fig.~\ref{fig:offgrid1_setup}.

\begin{figure}[h!]
\centering
\includegraphics[width=0.35\textwidth,trim={4.5cm 0 5cm 0},clip]{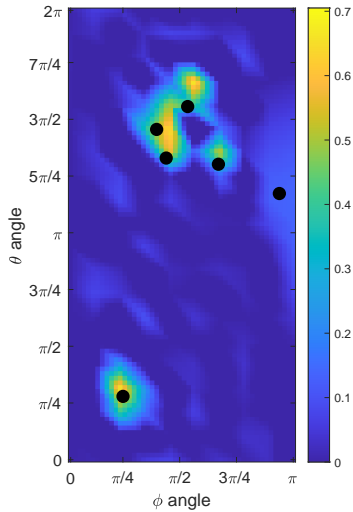}
\caption{ \label{fig:offgrid1_setup} The least squares reconstruction of the inverse source problem for $k=6$ sources.  The reconstruction was conducted over a $20\times 10$ grid in spherical parameters $(\theta,\phi)$.  The plotted result is an interpolation over this grid.  The true incident angles of the 6 sources are indicated by the black dots.  Note that the sources did not have identical amplitudes.  They were normally distributed with mean 1 and standard deviation 0.2.  Their true values in ascending order were 0.668, 0.861, 0.897, 0.944, 1.10, and 1.23. }
\end{figure}

We then performed the proposed off-grid recovery algorithm on this setup.  We emphasize that it was assumed that the true sparsity level of $k=6$ was unknown.  The thresholding was performed via \eqref{eq:Kt_def} with $c^{(t)}=1/4$ for all $t$.  A visualization of the algorithm's first iteration is shown below in Fig.~\ref{fig:offgrid1_iterations}.  We then ran the off-grid recovery algorithm on the identical setup, but first added noise to the data.  The noise added was Gaussian white noise at a level of 1\% with respect to the data.  

\begin{figure}[h!]
\centering
\begin{subfigure}[b]{0.35\textwidth}
                \centering
                \includegraphics[width=\textwidth,trim={4.5cm 0 5cm 0},clip]{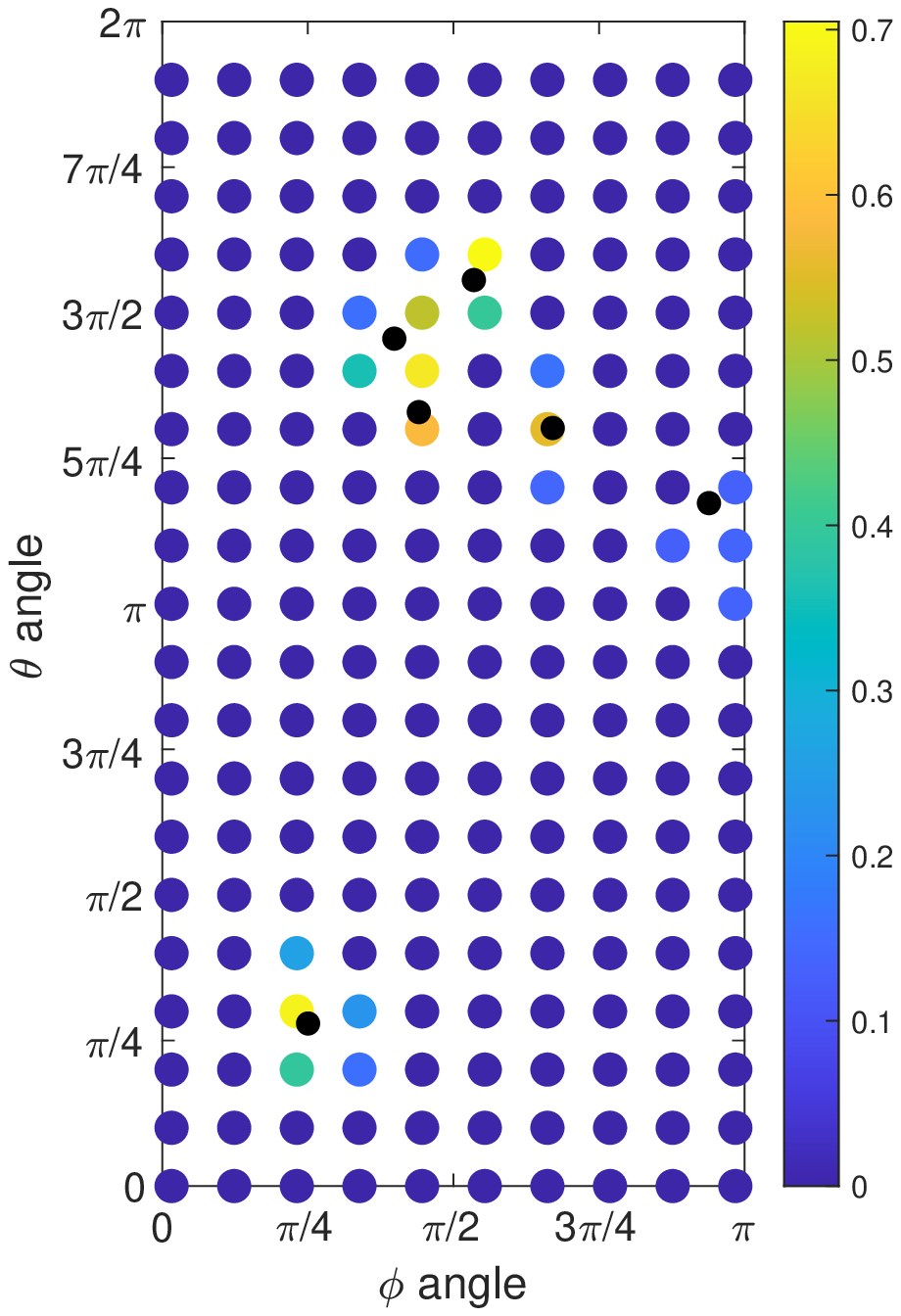}
                \caption{IHT on Initial Grid}
        \end{subfigure}
\begin{subfigure}[b]{0.35\textwidth}
                \centering
                \includegraphics[width=\textwidth,trim={4.5cm 0 5cm 0},clip]{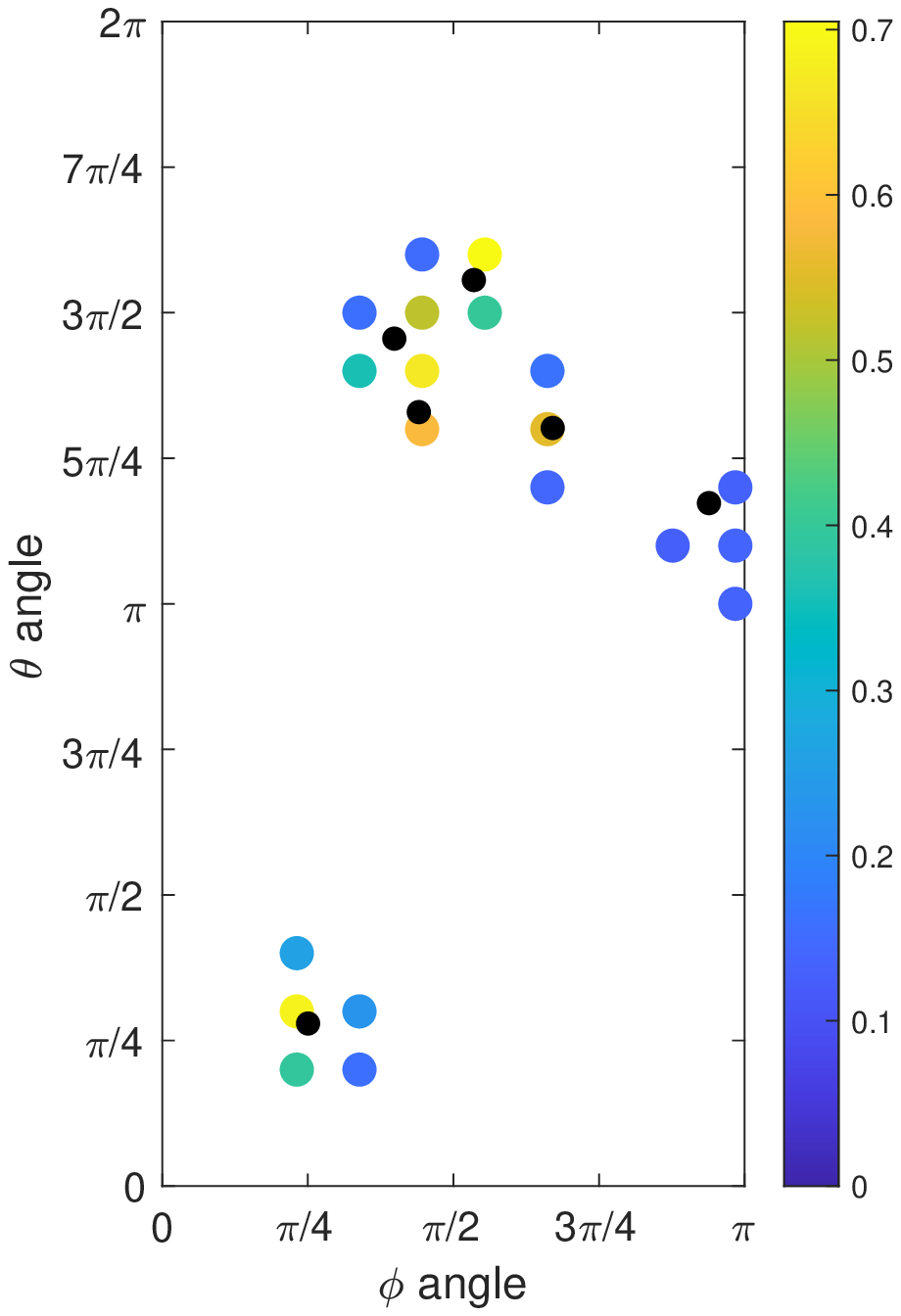}
                \caption{Keep Non-thresholded Points}
        \end{subfigure}
        \begin{subfigure}[b]{0.35\textwidth}
                \centering
                \includegraphics[width=\textwidth,trim={4.5cm 0 5cm 0},clip]{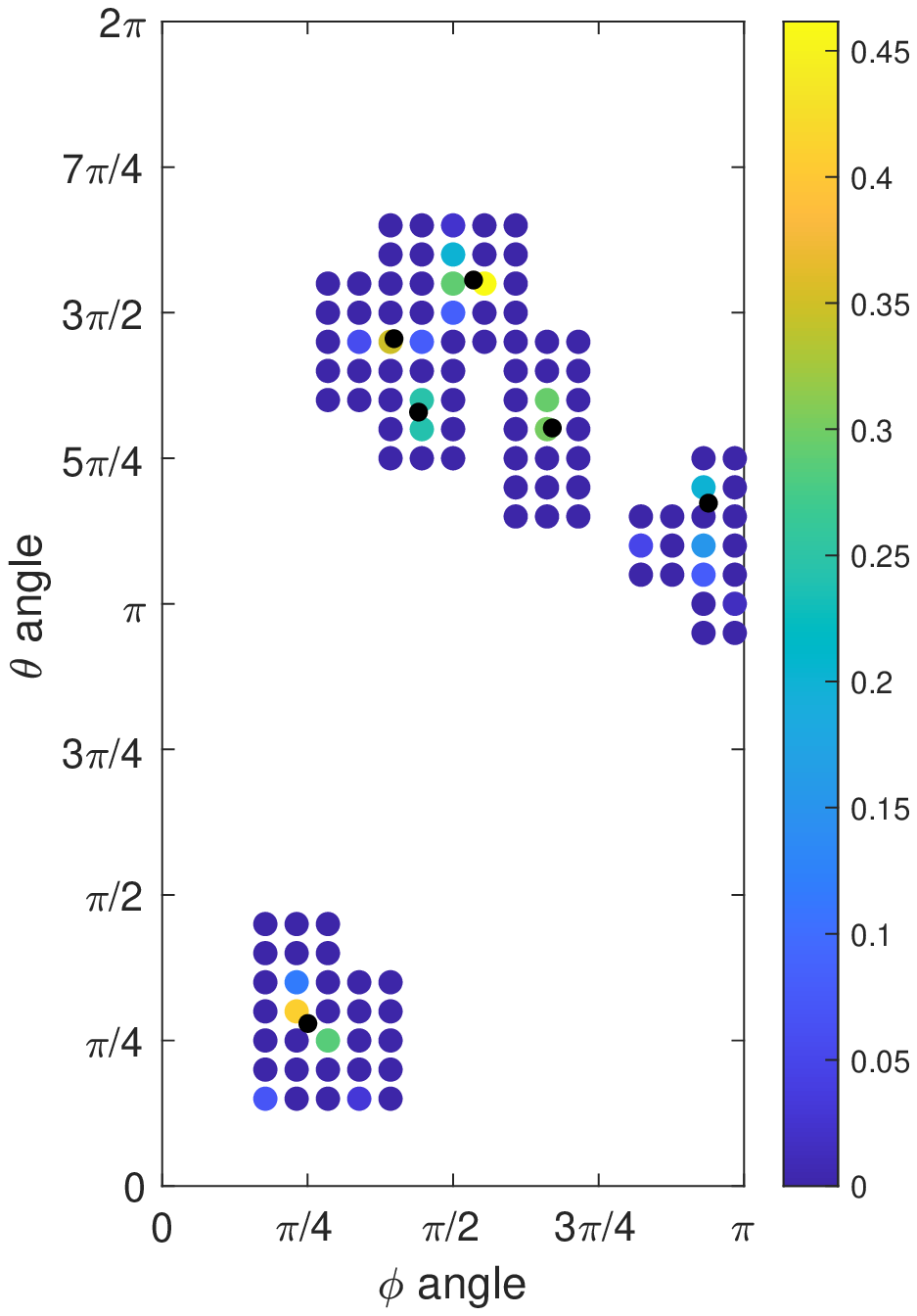}
                \caption{Structured IHT on New Grid}
        \end{subfigure}
        \begin{subfigure}[b]{0.35\textwidth}
                \centering
                \includegraphics[width=\textwidth,trim={4.5cm 0 5cm 0},clip]{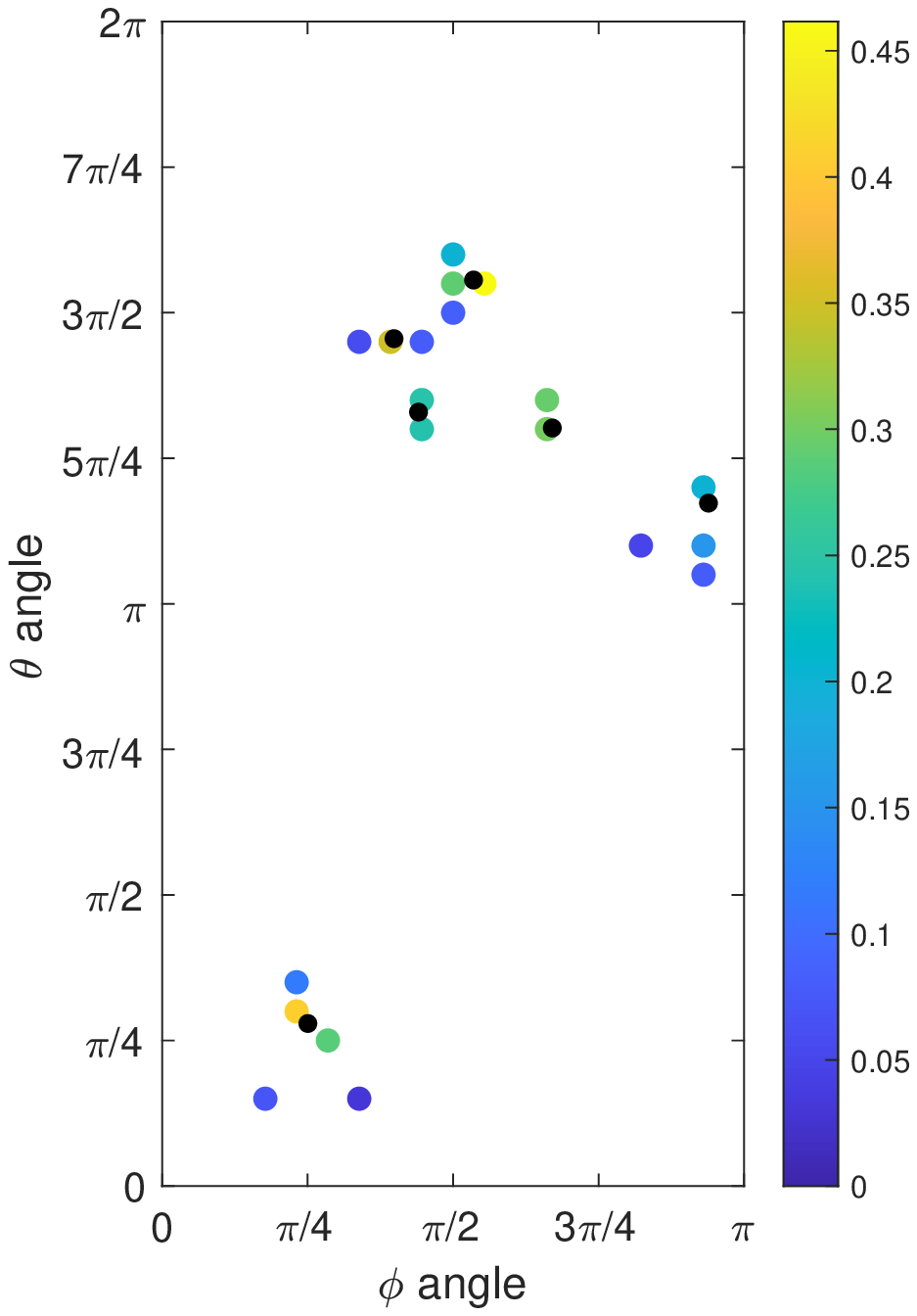}
                \caption{Keep Non-thresholded Points}
        \end{subfigure}
\caption{ \label{fig:offgrid1_iterations} Visualization of the first iteration through the off-grid recovery algorithm.  The process starts from an initial uniform grid of 200 candidate angles (A).  After running structured IHT on this setup, twenty angles were kept above the threshold (B).  In (C), the grid was refined about each of these grid points, and structured IHT was run again, thresholding all but the largest entry in each refined grid grouping.   We then thresholded and eliminated all points that were less than 1/4 the mean of all remaining values.  The steps shown in (C) and (D) were then iteratively repeated for 100 iterations in total.   }
\end{figure}

Two errors were measured to quantify the accuracy of recovery.  Let $a^{rec}_j$ and $\Theta^{rec}_j$ be the reconstructed amplitude and incident angle values.  We define the errors:
\begin{equation}
a_{err}^2 = \sum_{j=1}^6 |a^{rec}_j-a_j|^2 \qquad \ , \ \qquad \Theta_{err}^2 = \sum_{j=1}^6 |\Theta^{rec}_j-\Theta_j|^2 \ .
\end{equation}
Note that there was no guarantee that there were exactly 6 sources recovered by this algorithm.  However, this algorithm accurately found exactly the 6 sources.  For the noiseless experiment, $a_{err}=7.49\times10^{-6}$ and $\Theta_{err}=1.19\times10^{-5}$.  For the case with 1\% noise, the final error values increased to $a_{err}=4.10\times10^{-3}$ and $\Theta_{err}=1.22\times10^{-2}$, but were still quite accurate.  In subsequent tests, the algorithm was able to consistently and reliably recover up to $\sim$12 sources whose incident angles were separated by at least 0.05 units and whose amplitudes were normally distributed with mean 1 and standard deviation 0.2.  Overall, we do not claim that  our choice of parameters in the algorithm were optimal, but they worked reliably.  

In Fig.~\ref{fig:offgrid1_cohs}, we show how the coherence values change as the algorithm progresses, and why it is reasonable to expect convergence in some sense.  The coherence of the matrix $A(\mathbf{Y}^{(t)})$ each iteration is barely discernible from 1, and these values do not decrease as we progress through the algorithm.  This is to be expected, because we are refining the grid, and our candidate angles are getting closer together.  The same is true for any $\mu_{S_j^{(t)}}$.  However, the restricted coherence between two different sets improves as the algorithm hones in on the true locations.  The values of $\mu_{S_j^{(t)},S_{j'}^{(t)}}$ converge to 0.6835, which is the coherence of the matrix $A$ if we only input columns for the 6 true source angles.  One can see the quick drop-off in this behavior at the 33rd iteration in the noiseless case and at the 61st iteration in the case with noise.  At these iterations the algorithm correctly thresholds down to 6 source angles that are near the correct values.   

\begin{figure}[h!]
\centering
\includegraphics[width=0.6\textwidth]{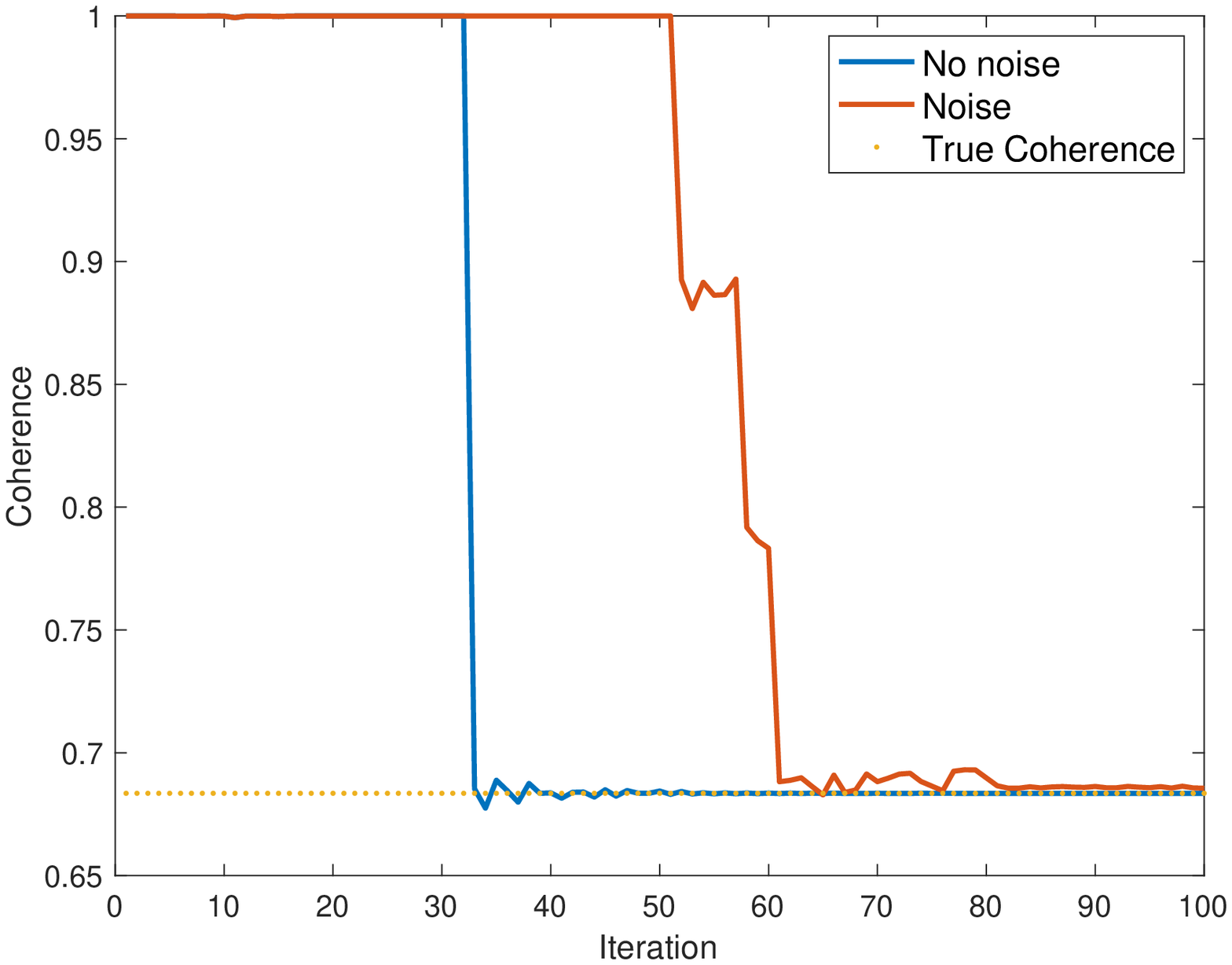}
\caption{ \label{fig:offgrid1_cohs} Plotted values of the restricted coherence $\mu_{S,S'}$ of the matrix $A(\mathbf{Y}^{(t)})$ after refining the grid.  This corresponds to picture (C) in Fig.~\ref{fig:offgrid1_iterations}.  The true coherence shown in the yellow dotted line at 0.6835 is the coherence of the matrix with 6 columns representing the 6 true incident angles of the sources.  One can see that for both the noise and noiseless cases, the coherence remains near 1 for a time.  However, once the algorithm eliminates extra grid points near the true incident angles, the coherence drops towards the optimal coherence and converges towards this value as the grid refines towards the true solution.}
\end{figure}

\section{Discussion and Future Works}

We have considered a variant of IHT with structured sparsity.   The convergence and error of the method were analyzed by means of coherence and compared to numerical simulations. When additional information is known in the form of structured sparsity, the analysis provides stronger and faster convergence guarantees than for IHT.  A related off-grid recovery algorithm was proposed that can overcome some of the limitations of the theory based on coherence.  Both algorithms performed well to address a specific inverse source problem.

The grid refinement and thresholding procedure for the off-grid recovery algorithms were determined by numerical testing.  However, it is unclear what would be optimal choices for these steps.  A thorough numerical study to optimize the off-grid algorithm will be considered.  One direction for future work is to compare the proposed methods (especially the off-grid algorithms) against other algorithms for DOA problems.  While this paper used the ISP as an example for demonstrating the use of the proposed algorithms, it would be worthwhile to benchmark structured IHT and Algorithm 3 against popular algorithms for solving ISPs.  This type of investigation could help develop the algorithm to find ideal parameters at each step.  
Accompanying this algorithm development, it would be beneficial to generalize the theory for Algorithm 2 to allow for more general grid refinement procedures.  This could also lead towards finding optimal parameters.

The structure of the sparsity considered in this paper was one specific definition.  It would be very interesting to extend the structured IHT algorithm (and its coherence theory) to apply to other definitions of structured sparsity.  Allowing for overlapping groups would be a natural next step.  Lastly, extending the structured IHT algorithm to solve nonlinear sparse recovery problems would be interesting.  A natural application would be the ISP with scattering.

\bibliographystyle{plain}
\bibliography{local}
\end{document}